\documentclass[12pt,oneside,english]{amsart}
\usepackage[T1]{fontenc}
\usepackage[latin1]{inputenc}
\usepackage{amsthm}
\usepackage{amstext}
\usepackage{amssymb}
\PassOptionsToPackage{normalem}{ulem}
\usepackage{ulem}

\makeatletter
\numberwithin{equation}{section}
\numberwithin{figure}{section}
\theoremstyle{plain}
\newtheorem{thm}{\protect\theoremname}[section]
  \theoremstyle{definition}
  \newtheorem{defn}[thm]{\protect\definitionname}
  \theoremstyle{remark}
  \newtheorem{rem}[thm]{\protect\remarkname}
\newcounter{mainthm}
\theoremstyle{plain}

\newtheorem{main_thm}[mainthm]{Main Theorem}
  \theoremstyle{plain}
  \newtheorem{cor}[thm]{\protect\corollaryname}
  \theoremstyle{plain}
  \newtheorem{lem}[thm]{\protect\lemmaname}
  \theoremstyle{plain}
  \newtheorem{fact}[thm]{\protect\factname}
  \theoremstyle{remark}
  \newtheorem{claim}[thm]{\protect\claimname}
  \theoremstyle{definition}
  \newtheorem{example}[thm]{\protect\examplename}
  \theoremstyle{plain}
  \newtheorem{prop}[thm]{\protect\propositionname}
  \newcounter{casectr}
  \newenvironment{caseenv}
  {\begin{list}{{\itshape\ \protect\casename} \arabic{casectr}.}{%
   \setlength{\leftmargin}{\labelwidth}
   \addtolength{\leftmargin}{\parskip}
   \setlength{\itemindent}{\listparindent}
   \setlength{\itemsep}{\medskipamount}
   \setlength{\topsep}{\itemsep}}
   \setcounter{casectr}{0}
   \usecounter{casectr}}
  {\end{list}}
  \theoremstyle{plain}
  \newtheorem*{prop*}{\protect\propositionname}

\PassOptionsToPackage{normalem}{ulem}
\usepackage{ulem}
\let\cite@rig\cite
\newcommand{\b@xcite}[2][\%]{\def\def@pt{\%}\def\pas@pt{#1}
  \mbox{\ifx\def@pt\pas@pt\cite@rig{#2}\else\cite@rig[#1]{#2}\fi}}
\renewcommand{\underbar}[1]{{\let\cite\b@xcite\uline{#1}}}

\usepackage[latin1]{inputenc}
\usepackage{a4wide}
\usepackage{amssymb}
\usepackage{url}
\makeatletter

\usepackage[all]{xy}

\makeatother

\makeatother

\usepackage{babel}
  \providecommand{\casename}{Case}
  \providecommand{\claimname}{Claim}
  \providecommand{\corollaryname}{Corollary}
  \providecommand{\definitionname}{Definition}
  \providecommand{\examplename}{Example}
  \providecommand{\factname}{Fact}
  \providecommand{\lemmaname}{Lemma}
  \providecommand{\propositionname}{Proposition}
  \providecommand{\remarkname}{Remark}
\providecommand{\theoremname}{Theorem}

\begin{document}
\global\long\def\Aut{\operatorname{Aut}}
\global\long\def\id{\operatorname{id}}
\global\long\def\supp{\operatorname{supp}}
\global\long\def\Inn{\operatorname{Inn}}
\global\long\def\Ff{\mathbb{F}}
\global\long\def\Zz{\mathbb{Z}}
\global\long\def\Cc{\mathbb{C}}
\global\long\def\Ll{\mathbb{L}}
\global\long\def\Vv{\mathbb{V}}
\global\long\def\Qq{\mathbb{Q}}
\global\long\def\phigh{p\mbox{-high}}
\global\long\def\nlg{\operatorname{nlg}}
\global\long\def\nor{\operatorname{nor}}
\global\long\def\ord{\mathbf{ord}}
\global\long\def\leftexp#1#2{{\vphantom{#2}}^{#1}{#2}}
\global\long\def\isp#1{{\left[{#1}\right]}}

\title{Automorphism towers and automorphism groups of fields without Choice}

\dedicatory{Dedicated to Professor R\"udiger G\"obel on his 70th birthday}

\author{Itay Kaplan and Saharon Shelah}

\thanks{The second author would like to thank the United States-Israel Binational
Science Foundation for partial support of this research. Publication
913.}
\begin{abstract}
This paper can be viewed as a continuation of \cite{Sh882} that dealt
with the automorphism tower problem without Choice. Here we deal with
the inequality $\tau_{\kappa}^{\nlg}\leq\tau_{\kappa}$ without Choice
and introduce a new proof to a theorem of Fried and Koll\'ar that
any group can be represented as an automorphism group of a field.
The proof uses a simple construction: working more in graph theory,
and less in algebra. 
\end{abstract}
\maketitle

\section{introduction and preliminaries}

\subsection*{Background}

Although this paper hardly mentions automorphism towers, it is the
main motivation for it. So we shall start by giving the story behind
them.

Given any centerless group $G$, $G\cong\Inn\left(G\right)\leq\Aut\left(G\right)$
so we can embed $G$ into its automorphism group. Also, an easy exercise
shows that $\Aut\left(G\right)$ is also without center, so we can
do this again, and again:
\begin{defn}
\label{TauG}For a centerless group $G$, we define \emph{the automorphism
tower} $\left\langle G^{\alpha}\left|\,\alpha\in\mathbf{ord}\right.\right\rangle $
by 
\begin{itemize}
\item $G^{0}=G$. 
\item $G^{\alpha+1}=\Aut\left(G^{\alpha}\right)$. 
\item $G^{\delta}=\cup\left\{ G^{\alpha}\left|\,\alpha<\delta\right.\right\} $
for $\delta$ limit. 
\end{itemize}
\end{defn}
\begin{rem}
The union in limit stages can be understood as the direct limit. But
we shall think of the tower as an increasing continuous sequence of
groups.
\end{rem}
The natural question that arises, is whether this process stabilizes,
and when. We define
\begin{defn}
For such a group, define $\tau_{G}=\min\left\{ \alpha\left|\, G^{\alpha+1}=G^{\alpha}\right.\right\} $. 
\end{defn}
In 1939, Weilandt proved in \cite{Wielandt} that for finite $G$,
$\tau_{G}$ is finite. What about infinite $G$? There exist examples
of centerless infinite groups such that this process does not stop
in any finite stage. For example --- the infinite dihedral group $D_{\infty}=\left\langle x,y\left|\, x^{2}=y^{2}=1\right.\right\rangle $
satisfies $\Aut\left(D_{\infty}\right)\cong D_{\infty}$ while the
automorphism replacing $x$ with $y$ is not in $\Inn\left(D_{\infty}\right)$.
The question remained open until the works of Faber \cite{Faber}
and Thomas \cite{ThomasI,ThomasII} (who was not aware of Faber's
work), that showed $\tau_{G}<\left(2^{\left|G\right|}\right)^{+}$.
\begin{defn}
For a cardinal $\kappa$ we define $\tau_{\kappa}$ as the smallest
ordinal such that $\tau_{\kappa}>\tau_{G}$ for all centerless groups
$G$ of cardinality $\leq\kappa$, or in other words 
\[
\tau_{\kappa}=\bigcup\left\{ \tau_{G}+1\left|\, G\mbox{ is centerless and }\left|G\right|\leq\kappa\right.\right\} .
\]

\end{defn}
Since $\left(2^{\kappa}\right)^{+}$ is regular we can immediately
conclude $\tau_{\kappa}<\left(2^{\kappa}\right)^{+}$. 

This paper is concerned with a Choiceless universe, i.e. we do not
assume the axiom of Choice. As a consequence, the previous definition
is generalized to
\begin{defn}
For a set $k$, we define $\tau_{\left|k\right|}$ to be the smallest
ordinal $\alpha$ such that $\alpha>\tau_{G}$ for all groups $G$
with power $\leq\left|k\right|$.
\end{defn}
Note that when we write $\left|X\right|\leq\left|Y\right|$ as in
the definition above, we mean that there is an injective  function
from $X$ to $Y$. Below we provide a short glossary.

A helpful and close notion is that of \emph{the normalizer tower}
$\left\langle \nor_{G}^{\alpha}\left(H\right)\left|\,\alpha\in\ord\right.\right\rangle $
of a subgroup $H$ of $G$ in $G$.
\begin{defn}
Let
\begin{itemize}
\item $\nor_{G}^{0}\left(H\right)=H$. 
\item $\nor_{G}^{\alpha+1}\left(H\right)=\nor_{G}\left(\nor_{G}^{\alpha}\left(H\right)\right)$. 
\item $\nor_{G}^{\delta}\left(H\right)=\bigcup\left\{ \nor_{G}^{\alpha}\left(H\right)\left|\,\alpha<\delta\right.\right\} $
for $\delta$ limit. 
\end{itemize}
And we let the normalizer length be $\tau_{G,H}^{\nlg}=\min\left\{ \alpha\left|\,\nor_{G}^{\alpha+1}\left(H\right)=\nor_{G}^{\alpha}\left(H\right)\right.\right\} $
(sometimes we just write $\tau_{G,H}$). 
\end{defn}
Analogously to $\tau_{\kappa}$, we define
\begin{defn}
\label{TauNLG}For a cardinal $\kappa$, let $\tau_{\kappa}^{\nlg}$
be the smallest ordinal such that $\tau_{\kappa}^{\nlg}>\tau_{\Aut\left(\mathfrak{A}\right),H}$,
for every structure $\mathfrak{A}$ of cardinality $\leq\kappa$ and
$H\leq\Aut\left(\mathfrak{A}\right)$ of cardinality $\leq\kappa$. 

In general (i.e. without assuming Choice), for a set $k$, we define
$\tau_{\left|k\right|}^{\nlg}$ as the smallest ordinal $\alpha$,
such that for every structure $\mathfrak{A}$ of power $\left|\left|\mathfrak{A}\right|\right|\leq\left|k\right|$,
$\tau_{\Aut\left(\mathfrak{A}\right),H}<\alpha$ for every subgroup
$H\leq\Aut\left(\mathfrak{A}\right)=G$ of power $\left|H\right|\leq\left|k\right|$.
In other words, $\tau_{\left|k\right|}^{\nlg}=\sup\left\{ \tau_{G,H}+1\left|\,\mbox{ for such }G,H\right.\right\} $.
\end{defn}
In \cite[Lemma 1.8]{JuShTh}, Just, Shelah and Thomas proved the following
inequality 
\[
\tau_{\kappa}\geq\tau_{\kappa}^{\nlg}.
\]
In fact it was essentially already proved by Thomas in \cite{ThomasI}. 

In \cite{Sh882} we dealt with an upper bound of $\tau_{\kappa}$
without assuming Choice. Here we prove $\tau_{\kappa}\geq\tau_{\kappa}^{\nlg}$
without Choice, and also provide a Choiceless variant of $\tau_{\left|k\right|}\geq\tau_{\left|k\right|}^{\nlg}$. 

It is worth mentioning some previous results regarding $\tau_{\kappa}$
that were proved using this inequality. 

In \cite{ThomasI}, Thomas proved that $\tau_{\kappa}\geq\kappa^{+}$.
It is a easy to conclude from Main Theorem \ref{thm:MainThm} below
that this result still holds without Choice. We will elaborate in
the end of this section (See Corollary \ref{cor:Lower Bound for Tau_kappa}). 

In \cite{JuShTh} the authors found that for uncountable $\kappa$
one cannot find an explicit upper bound for $\tau_{\kappa}$ better
than $\left(2^{\kappa}\right)^{+}$ in $ZFC$ (using set theoretic
forcing). In \cite{Sh810}, Shelah proved that if $\kappa$ is strong
limit singular of uncountable cofinality then $\tau_{\kappa}>2^{\kappa}$
(using results from PCF theory). In the proofs the authors construct
normalizer towers to find lower bound for $\tau_{\kappa}$, but we
did not check how much Choice was used. 

It remains an open question whether or not there exists a countable
centerless group $G$ such that $\tau_{G}\geq\omega_{1}$.

\subsection*{Description of paper}

As mentioned before, we wish to prove $\tau_{\kappa}\geq\tau_{\kappa}^{\nlg}$
without Choice. So we started by reading what was done in \cite{JuShTh}
(which is also described in detail in \cite{ThomasBook}).

The proof contains three parts:
\begin{enumerate}
\item \label{enu:StrToGrph}Given some structure, code it in a graph (i.e.
find a graph with the same cardinality and automorphism group). 
\item \label{enu:GrphToFld}Given a graph code it in a field. Now we have
a field $K$ with some subgroup $H\leq\Aut\left(K\right)$ such that
$\left|K\right|=\left|H\right|=\kappa$. 
\item \label{enu:Final}Use some lemmas from group theory and properties
of $PSL\left(2,K\right)$ to find a centerless group whose automorphism
tower coincides with the normalizer tower of $H$ in $\Aut\left(K\right)$. 
\end{enumerate}
Our first intention was to mimic this proof, and to prove some version
of $\tau_{\left|k\right|}\geq\tau_{\left|k\right|}^{\nlg}$ (see definitions
\ref{TauNLG} and \ref{TauG} above). To explain what we did prove,
we need some notation:
\begin{defn}
\label{def:TransX-1}Let $X$ be a set. 
\begin{enumerate}
\item $X^{<\omega}$ is the set of all finite sequences of members of $X$. 
\item $\left[X\right]^{<\aleph_{0}}=\left\{ a\subseteq X\left|\,\left|a\right|<\aleph_{0}\right.\right\} $.
\item $X^{\left\langle <\omega\right\rangle }=\left[X^{<\omega}\right]^{<\aleph_{0}}$,
i.e. the set of all finite subsets of finite sequences of elements
of $X$. 
\end{enumerate}
\end{defn}
Our methods cannot tackle $\tau_{\left|k\right|}\geq\tau_{\left|k\right|}^{\nlg}$
without Choice, since one often needs to code finite sequences. The
natural way to overcome this is to replace $k$ with $k^{<\omega}$,
so that we get $\tau_{\left|k^{<\omega}\right|}\geq\tau_{\left|k^{<\omega}\right|}^{\nlg}$.
However, we managed to proved a slightly different version:
\begin{main_thm}
\label{thm:MainThm}For any set $k$, $\tau_{\left|k^{\left\langle <\omega\right\rangle }\right|}\geq\tau_{\left|k^{\left\langle <\omega\right\rangle }\right|}^{\nlg'}$. 
\end{main_thm}
Where $\tau_{\left|k\right|}^{\nlg'}$ is a variant of $\tau_{\left|k\right|}^{\nlg}$.
See Definition \ref{def:Tau_Nlg1} below. 

With Choice there is no difference, and moreover, we get as a corollary
the original inequality for a cardinal $\kappa$ (see Corollary \ref{cor:CardMainThm}
below). It is a matter of taste whether replacing $k^{<\omega}$ and
$\nlg$ by $k^{\left\langle <\omega\right\rangle }$ and $\nlg'$
matters. Still, one may ask whether $\tau_{\left|k^{<\omega}\right|}^{\nlg}\leq\tau_{\left|k^{<\omega}\right|}$
or even $\tau_{\left|k\right|}\geq\tau_{\left|k\right|}^{\nlg}$ holds
without Choice.

Part (\ref{enu:StrToGrph}) was easy enough. However, it needs a passage
to a structure with countable language. This stage uses Choice. In
order to fix this, we just bypassed the problem all together and replaced
$\tau_{\left|k\right|}^{\nlg}$ by $\tau_{\left|k\right|}^{\nlg'}$.

Part (\ref{enu:Final}) was easy as well: An algebraic lemma which
obviously did not need Choice (Lemma \ref{lem:SimpleS}); And two
lemmas regarding $PSL\left(2,K\right)$ --- Lemma \ref{lem:PSLSimple}
and Lemma \ref{lem:(Van-der-Waerden)}. The latter is a theorem of
Van der Waerden and Schreier which described $\Aut\left(PSL\left(2,K\right)\right)$.
There is a simple model theoretic argument that shows that these lemmas
do not require Choice (Lemma \ref{lem:ModelTheoreticChoice}).

However, part (\ref{enu:GrphToFld}) seemed to be somewhat harder.
In \cite{JuShTh}, the authors referred to the work of Fried and Koll\'ar
\cite{FriedKollar}. In \cite{ThomasBook}, the author gives a less
technical proof that the construction in \cite{FriedKollar} works.
The proof, in both cases, was a little bit complicated, and we were
suspicious that Choice was used in it. After some time we realized
that it is most likely not used, but by then we already came up with
a proof of our own, in which the construction of the field is much
simpler, and thought that it is worth presenting. So, for part (\ref{enu:GrphToFld})
we prove:
\begin{main_thm}
\label{MainThm:Graphs2Fields}Let $\Gamma=\left\langle X,E\right\rangle $
be a connected graph. Then for any choice of characteristic there
exists a field $K_{\Gamma}$ of that characteristic such that $\left|K_{\Gamma}\right|\leq\left|X^{\left\langle <\omega\right\rangle }\right|$
and $\Aut\left(K_{\Gamma}\right)\cong\Aut\left(\Gamma\right)$. 
\end{main_thm}
The proof of Main Theorem \ref{MainThm:Graphs2Fields} is given in
Section \ref{sec:GraphsToFields}. Here we will give a brief outline
of the construction.

The plan was this: work a little bit on the graph, so that the algebra
would be easier. First code the given graph as a graph with the following
properties: its edges are colorable with some finite number $N$ of
colors, and the subgraphs induced by any particular color is a union
of disjoint stars. This is done in Lemma \ref{lem:Graph-color}. 

Now the construction of the field is as follows: first let $\left\langle p_{0},p_{1},\ldots,p_{N}\right\rangle $
be a list of distinct odd primes. Start with $\mathbb{Q}$ (or any
prime field), and add the set of vertices  $X$ as transcendental
elements over it. For each one, add $p_{0}^{n}$ roots to it for all
$n<\omega$. Now, for each edge, $e=\left\{ s,t\right\} $, colored
with the color $l<N$, adjoin $p_{l+1}^{n}$ roots for all $n<\omega$
to $\left(s+t+1\right)$. This is it. The reader is invited to compare
to \cite{FriedKollar}.

This construction can be done without Choice.

In the proof we use a generalized form of a lemma by P. Pr\"ohle
that appears in \cite{prohleEndomorphism}. In their original paper,
Fried and Koll\'ar could construct $K_{\Gamma}$ with the restriction
that $char\left(K_{\Gamma}\right)\neq2$ and Prohle removed this restriction.
His {}``third lemma'' from \cite{prohleEndomorphism} seemed to
be perfect for our situation. However, we needed to generalize it
in order to suit our purposes (and prove the generalization). This
is Lemma \ref{lem:MainLemma}. The proof of Lemma \ref{lem:MainLemma}
is similar to the one in \cite{prohleEndomorphism} and can be found
in Section \ref{sec:some-technical-lemmas}.

\subsection*{Acknowledgment}

We would like to thank the referee for many useful remarks and to
Haran Pilpel for drawing a graph with certain properties in record
time.

\subsection*{A note about reading this paper}

If the reader is not interested in Choice, but still wants to see
the proof of Main Theorem \ref{thm:MainThm} and Main Theorem \ref{MainThm:Graphs2Fields},
he should ignore all the computations of cardinalities, since they
become trivial. Also, with Choice, the construction of the field is
somewhat easier --- in our construction, we took the polynomial ring
$\Qq\left[Y\right]$ (where $Y$ is a set containing the vertices)
and then the quotient by an ideal. Then we had to show the ideal is
prime in order to take the field of fractions. But with Choice we
can construct the field by adding roots from the algebraic closure.
See also Remark \ref{rem:Why Choice is needed}.

\subsection*{A small glossary}
\begin{itemize}
\item $\left|X\right|\leq\left|Y\right|$ means: There is an injective function
from $X$ to $Y$. 
\item $\left|X\right|=\left|Y\right|$ means: There is a bijection from
$X$ onto $Y$. 
\item For a structure $\mathfrak{A}$, $\left|\mathfrak{A}\right|$ is its
universe and $\left|\left|\mathfrak{A}\right|\right|$ is its cardinality.
\item $\Vv$ is the universe and $\Ll$ is G\"odel's constructible universe.
\end{itemize}

\section{A variant of $\tau_{\left|k\right|}^{\nlg}$ and some corollaries
of Main Theorem \ref{thm:MainThm}}
\begin{defn}
\label{def:Rigid}A structure $\mathfrak{A}$ is called rigid if $\Aut\left(\mathfrak{A}\right)=1$,
i.e. it has no non-trivial automorphism. 
\end{defn}

\begin{defn}
\label{def:Tau_Nlg1}For a set $k$, we define $\tau_{\left|k\right|}^{\nlg'}$
as the smallest ordinal $\alpha$ which is greater than $\tau_{\Aut\left(\mathfrak{A}\right),H}$
where $\mathfrak{A},H$ are as in Definition \ref{TauNLG} and in
addition the vocabulary (language) $L$ of $\mathfrak{A}$ satisfies
\begin{enumerate}
\item There is some rigid structure with universe $L$ and a countable vocabulary
(for instance, $L$ is well-orderable); and 
\item $\left|L\right|\leq\left|\left|\mathfrak{A}\right|^{\left\langle <\omega\right\rangle }\right|$. 
\end{enumerate}
\end{defn}
\begin{rem}
\label{rem: Tau =00003D Tau '}If $\kappa$ is a cardinal number (i.e.
an $\aleph$), then $\tau_{\kappa}^{\nlg}=\tau_{\left|\kappa^{\left\langle <\omega\right\rangle }\right|}^{\nlg'}$
and $\tau_{\kappa}=\tau_{\left|\kappa^{\left\langle <\omega\right\rangle }\right|}$.
This is true since $\left|\kappa^{\left\langle <\omega\right\rangle }\right|=\left|\kappa\right|$,
and because given any $\mathfrak{A}$ as in the definition, we may
assume that $\left|\mathfrak{A}\right|\subseteq\kappa$ and that $L$
is $\left|\mathfrak{A}\right|^{<\omega}\subseteq\kappa^{<\omega}$
which is well-orderable (see \cite[Observation 2.3]{Sh882}). 
\end{rem}
Hence, by Main Theorem \ref{thm:MainThm}
\begin{cor}
\label{cor:CardMainThm}($ZF$) For a cardinal $\kappa$, $\tau_{\kappa}^{\nlg}\leq\tau_{\kappa}$. 
\end{cor}
The following is another easy conclusion of Main Theorem \ref{thm:MainThm}
\begin{cor}
\label{cor:Lower Bound for Tau_kappa} ($ZF$) for any cardinal $\kappa$,
$\tau_{\kappa}\geq\kappa^{+}$. Moreover, letting $\upsilon_{k^{\left\langle <\omega\right\rangle }}$
be the smallest nonzero ordinal $\alpha$ such that there is no injective
function $f:\alpha\to k^{\left\langle <\omega\right\rangle }$, then
\textup{$\tau_{\left|k^{\left\langle <\omega\right\rangle }\right|}\geq\upsilon_{k^{\left\langle <\omega\right\rangle }}$
for any set $k$. }\end{cor}
\begin{proof}
By \cite{ThomasI}, we know that this result is true with Choice.
Moreover, he proves that $\tau_{\kappa}^{\nlg}\geq\kappa^{+}$ (see
Lemma in the proof of Theorem 2 there). Let $\alpha<\upsilon_{k^{\left\langle <\omega\right\rangle }}$
be some ordinal. We know that $\Ll\models\tau_{\left|\alpha\right|}^{\nlg}\geq\left|\alpha\right|^{+}>\alpha$
and that $\left|\alpha\right|\leq k^{\left\langle <\omega\right\rangle }$.

For a moment we work in $\Ll$. So there is a group $G$ (the automorphism
group of some structure) and a subgroup group $H\leq G$ such that
$\left|H\right|\leq\left|\alpha\right|$ and $\alpha\leq\tau_{G,H}$.
We may assume that $\left|G\right|\leq\left|\alpha\right|$. For one
reason, this is the way it is constructed in \cite{ThomasI}. However,
we give a self-contained explanation: 

Let $L$ be the language $\left\{ P,Q,<,R\right\} \cup L_{\mbox{Groups}}$
where $P,Q$ are predicates, $<,R$ are binary relation symbols and
$L_{\mbox{Groups}}$ is the language of groups. Consider the $L$-structure
$\mathfrak{G}$ with universe the disjoint union of $G$ and $\alpha$
where $P^{\mathfrak{G}}=G$, $Q^{\mathfrak{G}}=\alpha$, with the
group structure on $P$, the order on $Q$ and $R^{\mathfrak{G}}\left(x,\beta\right)$
holds iff $x\in\nor_{G}^{\beta}\left(H\right)$. Let $\mathfrak{G}'\prec\mathfrak{G}$
be an elementary substructure of size $\leq\left|\alpha\right|$ such
that $H\subseteq P^{\mathfrak{G}'}$, $\alpha\subseteq Q^{\mathfrak{G}'}$
(so $\alpha=Q^{\mathfrak{G}'}$), and let $G'=P^{\mathfrak{G}'}$.
As a group $G'$ is a subgroup of $G$ containing $H$ of size $\leq\left|\alpha\right|$
and for all $\beta<\alpha$, $\nor_{G'}^{\beta}\left(H\right)\neq\nor_{G'}^{\beta+1}\left(H\right)$,
and in particular $\alpha\leq\tau_{G',H}$. 

Now we go back to $\Vv$, so $\left|G\right|\leq\left|\alpha\right|\leq\left|k^{\left\langle <\omega\right\rangle }\right|$
by assumption. By \cite[Claim 2.8]{Sh882}, $\alpha\leq\tau_{G,H}^{\Ll}=\tau_{G,H}^{\Vv}$.
Let $\mathfrak{A}$ be the structure with universe $G$ and for each
$g\in G$ a unary function $f_{g}$ taking $x$ to $x\cdot g$. Then
$\Aut\left(\mathfrak{A}\right)\cong G$. So we conclude that $\tau_{k^{\left\langle <\omega\right\rangle }}^{\nlg'}\geq\alpha$
(because $G$ is well-orderable as in Remark \ref{rem: Tau =00003D Tau '}
above). By Main Theorem \ref{thm:MainThm}, $\tau_{\left|k^{\left\langle <\omega\right\rangle }\right|}\geq\alpha$. 
\end{proof}

\section{Coding structures as graphs}

The next lemma allows us to present any automorphism group of an (almost)
arbitrary structure as the automorphism group of a graph.
\begin{lem}
\label{lem:LangCoutable}Let $\mathfrak{A}$ be a structure for the
vocabulary (=language) $L$ such that \end{lem}
\begin{enumerate}
\item There is some rigid structure on $L$ with vocabulary $L'$ such that
$\left|L'\right|\leq\aleph_{0}$. 
\item $\left|L\right|\leq\left|\left|\mathfrak{A}\right|^{\left\langle <\omega\right\rangle }\right|$. 
\end{enumerate}
Then there is a structure $\mathfrak{B}$ with vocabulary $L_{\mathfrak{B}}$
such that 
\begin{itemize}
\item $\left|\left|\mathfrak{B}\right|\right|\leq\left|\left|\mathfrak{A}\right|\right|+\left|L\right|$
(so $\leq\left|\left|\mathfrak{A}\right|^{\left\langle <\omega\right\rangle }\right|$)
\item $\Aut\left(\mathfrak{B}\right)\cong\Aut\left(\mathfrak{A}\right)$ 
\item $\left|L_{\mathfrak{B}}\right|=\aleph_{0}$\end{itemize}
\begin{proof}
We may assume that both $L$ and $L'$ are relational languages.

Define $\mathfrak{B}$ by: 
\begin{itemize}
\item $\left|\mathfrak{B}\right|=\mathfrak{\left|A\right|}\times\left\{ 0\right\} \cup L\times\left\{ 1\right\} $. 
\item The vocabulary is $L_{\mathfrak{B}}=\left\{ R_{n}\left|\, n\in\omega\right.\right\} \cup L'\cup\left\{ P\right\} $
where $P$ is a unary predicate and each $R_{n}$ is an $n+1$ place
relation. 
\end{itemize}
Where: 
\begin{itemize}
\item $Q^{\mathfrak{B}}=Q^{L}$ on $L\times\left\{ 1\right\} $ for each
$Q\in L'$.
\item $R_{n}^{\mathfrak{B}}=\left\{ \left(\left(a_{0},0\right),\ldots,\left(a_{n-1},0\right),\left(R,1\right)\right)\left|\,\begin{array}{c}
R\in L\textrm{ is an }n\textrm{ place relation and }\\
\left(a_{0},\ldots,a_{n-1}\right)\in R^{\mathfrak{A}}
\end{array}\right.\right\} $ 
\item $P^{\mathfrak{B}}=L\times\left\{ 1\right\} $. 
\end{itemize}
It is easy to see that $\mathfrak{B}$ is as desired.
\end{proof}
This is well known:
\begin{thm}
\label{thm:CodingStrucAsGraphs}Let $\mathfrak{A}$ be a structure
for the first order language $L$ which is as in the conditions of
\ref{lem:LangCoutable}. Then there is a connected graph $\Gamma=\left\langle X_{\Gamma},E_{\Gamma}\right\rangle $
such that $\Aut\left(\Gamma\right)\cong\Aut\left(\mathfrak{A}\right)$,
and $\left|X_{\Gamma}\right|\leq\left|\left|\mathfrak{A}\right|\right|^{<\aleph_{0}}$. \end{thm}
\begin{proof}
For details see e.g. \cite[Lemma 4.2.2]{ThomasBook} or \cite[Thereom 5.5.1]{Hod}.
From the construction (which does not use Choice) described there,
one can deduce the part regarding the cardinality. The proof uses
the fact that we can reduce to structures with countable languages,
but this is exactly Lemma \ref{lem:LangCoutable}. 
\end{proof}

\section{Some group theory}
\begin{lem}
\label{lem:SimpleS}Let $S$ be a simple non-abelian group, and let
$G$ be a group such that $\Inn\left(S\right)\leq G\leq\Aut\left(S\right)$.
Then the automorphism tower of $G$ is naturally isomorphic to the
normalizer tower of $G$ in $\Aut\left(S\right)$. 
\end{lem}
The proof of this lemma can be found in \cite[Theorem 4.1.4]{ThomasBook}
(and, of course, it does not use Choice).

So we need a simple group. Recall
\begin{defn}
Let $K$ be a field, $n<\omega$, then: 
\begin{itemize}
\item $GL\left(n,K\right)$ is the group of invertible $n\times n$ matrices
over $K$. 
\item $PGL\left(n,K\right)=GL\left(n,K\right)/Z\left(GL\left(n,K\right)\right)$
(Here, $Z\left(GL\left(n,K\right)\right)$ is the group $K^{\times}\cdot I$
where $I$ is the identity matrix). 
\item $SL\left(n,K\right)=\left\{ x\in GL\left(n,K\right)\left|\,\det\left(x\right)=1\right.\right\} $.
\item $PSL\left(n,K\right)=SL\left(n,K\right)/Z\left(SL\left(n,K\right)\right)$
(The denominator is just $Z\left(GL\left(n,K\right)\right)\cap SL\left(n,K\right)$).
\end{itemize}
\end{defn}
\begin{fact}
$PSL\left(n,K\right)$ is a normal subgroup of $PGL\left(n,K\right)$. \end{fact}
\begin{lem}
\label{lem:PSLSimple}$PSL\left(2,K\right)$ is simple for any field
$K$ such that $\left|K\right|\geq3$. 
\end{lem}
The proof of this lemma can be found in many books, e.g. \cite{Rotman}.
It is also true in $ZF$, by the following Lemma and Claim:
\begin{lem}
\label{lem:ModelTheoreticChoice}Suppose $P$ is a claim, such that
$ZFC\vdash P$, and $\psi$ is a first order sentence (in some language)
such that $ZF\vdash$'$P$ is true iff $\psi$ does not have a model'.
Then $ZF\vdash P$. \end{lem}
\begin{proof}
If we have a model $\Vv$ of $ZF$, such that $\Vv\models\neg P$,
then $\psi$ has a model so cannot prove contradiction (there is no
use of Choice here). Hence $\psi$ is consistent in $\mathbb{L}=\mathbb{L}^{\Vv}$
as well. (If $\psi$ was not consistent in $\Ll$, then a proof of
a contradiction from $\psi$ would exist in $\Vv$ as well). Hence,
by G\"odel Completeness Theorem in $ZFC$, $\mathbb{L}\models\neg P$,
but $\mathbb{L}\models ZFC$ --- a contradiction.\end{proof}
\begin{claim}
\label{cla:FirstOrderSenPSL}There is a first order sentence $\psi$
such that $\psi$ has a model iff there is a field $K$, $\left|K\right|\geq3$
such that $PSL\left(2,K\right)$ is not simple. \end{claim}
\begin{proof}
Let $L$ be the language of fields with an extra $4$-ary relation
$H$, i.e. $L=\left\{ +,\cdot,0,1,H\right\} $. Let the sentence $\psi$
say that the universe is a field $K$ of size $\geq3$ and that $H\subseteq K^{4}$
is a normal subgroup of $SL\left(2,K\right)$ (after some choice of
coordinates), and that $H$ contains $Z\left(SL\left(2,K\right)\right)$
and also some element outside $Z\left(SL\left(2,K\right)\right)$.
\end{proof}
We close this section by showing one final algebraic fact holds over
$ZF$. Recall:
\begin{defn}
Given any two groups $N$ and $H$ and a group homomorphism $\varphi:H\to\Aut\left(N\right)$,
we denote by $N\rtimes_{\varphi}H$ (or simply $N\rtimes H$ if $\varphi$
is known) the semi-direct product of $N$ and $H$ with respect to
$\varphi$. 
\end{defn}
Note that for a field $K$, there are canonical homomorphisms $\Aut\left(K\right)\to\Aut\left(PSL\left(2,K\right)\right)$
and $\Aut\left(K\right)\to\Aut\left(PGL\left(2,K\right)\right)$. 
\begin{fact}
\label{lem:(Van-der-Waerden)}(Van der Waerden, Schreier \cite{SchVanDer})
Let $K$ be a field. Then every automorphism of $PSL\left(2,K\right)$
is induced via conjugation by a unique element of $P\Gamma L\left(2,K\right):=PGL\left(2,K\right)\rtimes\Aut\left(K\right)$.
Hence $\Aut\left(PSL\left(2,K\right)\right)\cong P\Gamma L\left(2,K\right)$. 
\end{fact}
This means that if $\varphi\in\Aut\left(PSL\left(2,K\right)\right)$
then there are unique $\alpha\in\Aut\left(K\right)$ and $g\in PGL\left(2,K\right)$
such that for every $x\in PSL\left(2,K\right)$, $\varphi\left(x\right)=g\alpha\left(x\right)g^{-1}$. 

We again use the model theoretic argument of Lemma \ref{lem:ModelTheoreticChoice}
to give a proof of this fact in $ZF$:
\begin{claim}
$ $ 
\begin{enumerate}
\item \label{enu:Existance}There is a first order sentence $\psi$ such
that $\psi$ has a model iff there is a field $K$, and an automorphism
$\varphi\in\Aut\left(PSL\left(2,K\right)\right)$ such that $\varphi$
is not in $P\Gamma L\left(2,K\right)$. (This implies the existence
of $\left(\alpha,g\right)$ required by the fact). 
\item \label{enu:Uniquness}There is a first order sentence $\psi'$ such
that $\psi'$ has a model iff there is a field $K$, and some $1\neq g\in PGL\left(2,K\right)$,
$\alpha\in\Aut\left(K\right)$, such that for every $x\in PSL\left(2,K\right)$,
$\alpha\left(x\right)=gxg^{-1}$. (This implies the uniqueness of
$\left(\alpha,g\right)$ required by the fact). 
\end{enumerate}
\end{claim}
\begin{proof}
(1): Let $K$ be a field. Recall that $x_{t}=\left(\begin{array}{cc}
1 & t\\
0 & 1
\end{array}\right)$ and $z_{t}=\left(\begin{array}{cc}
1 & 0\\
t & 1
\end{array}\right)$ generate $SL(2,K)$. Let $g\in PGL\left(2,K\right)$, $\sigma\in\Aut\left(PSL\left(2,K\right)\right)$.

Then $\alpha\in\Aut\left(K\right)$ satisfies $\sigma\left(x\right)=g\alpha\left(x\right)g^{-1}$
iff the map $x\mapsto g^{-1}\sigma\left(x\right)g$ takes $\bar{x}_{t}$
to $\bar{x}_{\alpha\left(t\right)}$ and $\bar{z}_{t}$ to $\bar{z}_{\alpha\left(t\right)}$.
Let $L$ be the language of fields augmented with $4$-place function
symbols $\left\{ \sigma_{i}\left|\, i<4\right.\right\} $. $\psi$
says that the universe $K$ is a field, and that $\sigma$ is an automorphism
of $PSL\left(2,K\right)$ ($SL\left(2,K\right)$ is a definable subset
of $K^{4}$, as is $Z\left(SL\left(2,K\right)\right)$), such that
for all $g\in PGL\left(2,K\right)$, the maps $t\mapsto g^{-1}\sigma\left(\bar{x}_{t}\right)g$
and $t\mapsto g^{-1}\sigma\left(\bar{z}_{t}\right)g$ do not induce
a well defined automorphism of $K$.

(2): Let $L$ be the language of fields. $\psi'$ says that the universe
$K$ is a field and that there is some nontrivial $g\in PGL\left(2,K\right)$
such that the maps $t\mapsto g^{-1}\bar{x}_{t}g$ and $t\mapsto g^{-1}\bar{z}_{t}g$
are induced by an automorphism $\alpha$ of $K$.
\end{proof}

\section{Proof of Main Theorem \ref{thm:MainThm} from Main Theorem \ref{MainThm:Graphs2Fields}}

From Main Theorem \ref{thm:GraphToField} which is proved in the next
section, we can now deduce 

\setcounter{mainthm}{0}
\begin{main_thm}
For any set $k$, $\tau_{\left|k^{\left\langle <\omega\right\rangle }\right|}^{\nlg'}\leq\tau_{\left|k^{\left\langle <\omega\right\rangle }\right|}$. \end{main_thm}
\begin{proof}
(essentially the same proof as in \cite{JuShTh}). We are given a
structure $\mathfrak{A}$, with language $L$ such that on the set
$L$ there is a rigid structure with countable vocabulary, and $\left|\left|\mathfrak{A}\right|\right|\leq\left|k^{\left\langle <\omega\right\rangle }\right|$.
By Theorem \ref{thm:CodingStrucAsGraphs} and Main Theorem \ref{thm:GraphToField}
we may assume that $\mathfrak{A}$ is an infinite field, $K$. We
are also given a subgroup $H\leq\Aut\left(K\right)$, $\left|H\right|\leq\left|k^{\left\langle <\omega\right\rangle }\right|$. 

Let $G=PGL\left(2,K\right)\rtimes H$. Obviously $\left|G\right|\leq\left|k^{\left\langle <\omega\right\rangle }\right|$.

$G$ is centerless, because by Fact \ref{lem:(Van-der-Waerden)},
the centralizer of $PSL\left(2,K\right)$ in $P\Gamma L\left(2,K\right)$
is trivial, and $PSL\left(2,K\right)\leq G$. So $PSL\left(2,K\right)\leq G\leq P\Gamma L\left(2,K\right)$.
By Lemmas \ref{lem:SimpleS}, \ref{lem:PSLSimple}, and \ref{lem:(Van-der-Waerden)},
$G^{\alpha}$ is isomorphic to $\nor_{P\Gamma L\left(2,K\right)}^{\alpha}\left(G\right)$.

Now, by induction on $\alpha$, one has $\nor_{P\Gamma L\left(2,K\right)}^{\alpha}\left(G\right)=PGL\left(2,K\right)\rtimes\nor_{\Aut\left(K\right)}^{\alpha}\left(H\right)$
and we are done.
\end{proof}

\section{\label{sec:GraphsToFields}coding graphs as fields}

In the introduction we mentioned that the following theorem of Fried
and Koll\'ar \cite{FriedKollar} was used in \cite{JuShTh}:
\begin{thm}
(Fried and Koll\'ar) ($ZFC$) For every connected graph $\Gamma$
there is a field $K$ such that $\Aut\left(\Gamma\right)\cong\Aut\left(K\right)$,
and $\left|K\right|=\left|\Gamma\right|+\aleph_{0}$. 
\end{thm}
Here we will offer a different proof of the Choiceless version, namely
\begin{main_thm}
\label{thm:GraphToField}Let $\Gamma=\left\langle X,E\right\rangle $
be a connected graph. Then there exists a field $K_{\Gamma}$ of any
characteristic such that $\left|K_{\Gamma}\right|\leq\left|X^{\left\langle <\omega\right\rangle }\right|$
and $\Aut\left(K_{\Gamma}\right)\cong\Aut\left(\Gamma\right)$.\end{main_thm}
\begin{cor}
If $G$ is a group and there is some rigid structure with countable
vocabulary on it, then there is a field $K$ such that $\Aut\left(K\right)\cong G$,
and $\left|K\right|\leq\left|G^{\left\langle <\omega\right\rangle }\right|$.\end{cor}
\begin{proof}
(of corollary) Let $\mathfrak{A}$ be the structure with universe
$G$ and for each $g\in G$ a unary function $f_{g}$ taking $x$
to $x\cdot g$ so that $\Aut\left(\mathfrak{A}\right)\cong G$. Now
apply \ref{thm:CodingStrucAsGraphs} and Main Theorem \ref{thm:GraphToField}.
\end{proof}

\subsection{Coding graphs as colored graphs}

We start by working a bit on the graph, to make the algebra easier.
\begin{defn}
\label{def:star}A graph $G=\left\langle X,E\right\rangle $ is called
a star if there is a vertex $v$ (the center) such that $E\subseteq\left\{ \left\{ v,u\right\} \left|\, u\in V-\left\{ v\right\} \right.\right\} $. \end{defn}
\begin{lem}
\label{lem:Graph-color}There is some number $N$ such that for every
connected graph $\Gamma=\left\langle X_{\Gamma},E_{\Gamma}\right\rangle $,
there is a connected graph $\Gamma^{+}=\left\langle X_{\Gamma^{+}},E_{\Gamma^{+}}\right\rangle $
with the following properties: 
\begin{enumerate}
\item \label{enu:AutGamma}$\Aut\left(\Gamma\right)\cong\Aut\left(\Gamma^{+}\right)$. 
\item \label{enu:Coloring}There is a coloring $C:E_{\Gamma^{+}}\to N$
of the edges of $\Gamma^{+}$ in $N$ colors such that for all $l<N$
the $l$-th colored subgraph is a disjoint union of stars.
\item \label{enu:Preservation}Every $\varphi\in\Aut\left(\Gamma^{+}\right)$
preserves the coloring. 
\item \label{enu:Card}$\left|X_{\Gamma^{+}}\right|\leq\left|X_{\Gamma}^{\left\langle <\omega\right\rangle }\right|$,
in fact $\left|X_{\Gamma}\right|\leq\left|X_{\Gamma^{+}}\right|\leq\left|X_{\Gamma}\right|+4\left|E_{\Gamma}\right|$. 
\end{enumerate}
\end{lem}
\begin{proof}
The idea is to replace each edge $\left\{ x,y\right\} $ by a copy
of the graph $G$ described below.

Recall that the valency of a vertex is the number of edges incident
to the vertex, and will be denoted by $val\left(x\right)$. Let $G=\left\langle X_{G},E_{G}\right\rangle $
be the following auxiliary graph:

 
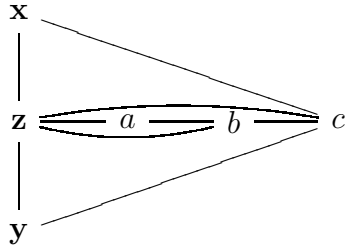
\[\xymatrix{\mathbf{x}\ar@{-}[rrrd]\\\mathbf{z}\ar@/_/@{-}[rr]\ar@/^/@{-}[rrr]\ar@{-}[d]\ar@{-}[u]\ar@{-}[r]&a\ar@{-}[r]&b\ar@{-}[r]&c\\\mathbf{y}\ar@{-}[rrru]}\]  

Note the following properties of $G$: 
\begin{itemize}
\item It has only $2$ automorphisms: $\id$ and $\sigma$, where $\sigma$
switches $\mathbf{x}$ and $\mathbf{y}$, but fixes all other vertices:
$\mathbf{z},b,c$ are characterized by their valency and $a$ is the
only vertex with valency $2$ which is adjacent to $b,\mathbf{z}$.
\item $\mathbf{z}$ is adjacent to all the vertices, its valency is unique
and is not divisible by $val\left(\mathbf{x}\right)$.
\item $\mathbf{x}$ and $\mathbf{y}$ are not adjacent. 
\end{itemize}
Description of $\Gamma^{+}$:

The set of vertices is

\[
X_{\Gamma^{+}}=\left\{ \left(1,x\right)\left|\, x\in X_{\Gamma}\right.\right\} \cup\left\{ \left(2,u,w\right)\left|\, u\in E_{\Gamma},w\in X_{G}-\left\{ \mathbf{x,y}\right\} \right.\right\} .
\]

And the edges are: 
\begin{itemize}
\item $\left(2,u,w\right)$ and $\left(2,u',w'\right)$ are adjacent iff
$u=u'$ and $\left\{ w,w'\right\} \in E_{G}$. 
\item $\left(1,x\right)$ and $\left(2,u,w\right)$ are adjacent iff $x\in u$
and $\left\{ \mathbf{x},w\right\} \in E_{G}$ (iff $\left\{ \mathbf{y},w\right\} \in E_{G}$). 
\item That is all. 
\end{itemize}
So, for each edge $\left\{ x,y\right\} =u\in E_{\Gamma}$ there is
an induced subgraph $\Gamma_{u}^{+}$ of $\Gamma^{+}$, whose vertices
are $\left\{ \left(1,x\right),\left(1,y\right)\right\} \cup\left\{ \left(2,u,w\right)\left|\, w\neq\mathbf{x},\mathbf{y}\right.\right\} $,
and $\Gamma_{\left\{ x,y\right\} }^{+}\cong G$ (by sending $\mathbf{x}$
to $\left(1,x\right)$, $\mathbf{y}$ to $\left(1,y\right)$ and $w\neq\mathbf{x},\mathbf{y}$
to $\left(2,u,w\right)$). 

Let $G'$ be the subgraph of $G$ induced by removing $\mathbf{y}$,
let $N=\left|E_{G'}\right|$ (so $N=7$), and denote $E_{G'}=\left\{ e_{0},\ldots,e_{N-1}\right\} $.
Let $f:\Gamma^{+}\to G'$ be a homomorphism of graphs defined as follows:
$f\left(1,x\right)=\mathbf{x},\, f\left(2,u,w\right)=w$. The coloring
$C:E_{\Gamma^{+}}\to N$ is defined by $C\left(e\right)=i$ iff $f\left(e\right)=e_{i}$.

Let us now show (\ref{enu:Coloring}). For each $i<N$, let $\Gamma_{i}^{+}=\left\langle X_{i},E_{i}\right\rangle $
be the subgraph induced by the color $i$. If $\mathbf{x}\notin e_{i}$,
then $\Gamma_{i}^{+}$ is a union of disjoint edges by the definitions
(and an edge is a star). If $\mathbf{x}\in e_{i}$, then $\Gamma_{i}^{+}$
is a disjoint union of $\left|X_{\Gamma}\right|$ stars, with centers
$\left\{ \left(1,x\right)\left|\, x\in X_{\Gamma}\right.\right\} $,
each having $val_{\Gamma}\left(x\right)$ edges.

For (\ref{enu:AutGamma}), note that $val_{\Gamma^{+}}\left(1,x\right)=val_{G}\left(\mathbf{x}\right)\cdot val_{\Gamma}\left(x\right)$
(or $\infty$, if $val_{\Gamma}\left(x\right)\geq\aleph_{0}$), while
$val_{\Gamma^{+}}\left(2,u,w\right)=val_{G}\left(w\right)$, hence
$val_{\Gamma^{+}}\left(2,u,\mathbf{z}\right)$ is not divisible by
$val_{\Gamma^{+}}\left(1,x\right)$. 

Hence if $\varphi\in\Aut\left(\Gamma^{+}\right)$ then $\varphi\left(2,u,\mathbf{z}\right)=\left(2,u',\mathbf{z}\right)$
for some $u'\in E_{\Gamma}$. Since $\mathbf{z}$ is adjacent to all
the vertices in $G$, $\Gamma_{\left\{ x,y\right\} }^{+}$ consists
of all the vertices $\left(2,u,\mathbf{z}\right)$ is adjacent to
and itself. So $\varphi\upharpoonright\Gamma_{u}^{+}$ is an isomorphism
onto $\Gamma_{u'}^{+}$. Since $\Aut\left(G\right)=\left\{ \id,\sigma\right\} $,
for all $w\neq\mathbf{x},\mathbf{y}$, $\varphi\left(2,u,w\right)=\left(2,u',w\right)$.
This allows us to define $\psi_{\varphi}=\psi\in\Aut\left(\Gamma\right)$
by $\psi\left(x\right)=x'$ where $\varphi\left(1,x\right)=\left(1,x'\right)$.
It is now easy to see that $\varphi\mapsto\psi_{\varphi}$ is an isomorphism
from $\Aut\left(\Gamma^{+}\right)$ onto $\Aut\left(\Gamma\right)$.

(\ref{enu:Preservation}) and (\ref{enu:Card}) should be clear. 
\end{proof}

\subsection{Coding colored graphs as fields}

Now we may assume that our graph is as in \ref{lem:Graph-color},
and we start constructing the field. 

We use the somewhat nonstandard notation of $r$ as the characteristic
of a field, so that $\Ff_{r}$ is the prime field with $r$ elements.
\begin{defn}
Let $F\subseteq K$ be a field extension. $F$ is said to be relatively
algebraically closed in $K$ if every $x\in K\backslash F$ is transcendental
over $F$. 
\end{defn}

\begin{defn}
\label{def:p-high}Let $p$ be a prime. An element $x$ in a field
$F$ is called \emph{$p$-high}, if there is a sequence $\left\langle x_{i}\left|\, i<\omega\right.\right\rangle $
of elements in $F$, such that $x_{0}=x$, and $x_{i+1}^{p}=x_{i}$.
With Choice this means that $x$ has a $p^{n}$-th root for all $n<\omega$.\end{defn}
\begin{example}
If $F=\Qq$, then for $p$ odd, the only $\phigh$ element in $F$
are $1,-1,0$. If $F=\Ff_{r}$ for some prime $r$, then for every
$p$ such that $\left(p,r-1\right)=1$ (i.e. the map $x\mapsto x^{p}$
is onto), every element in $F$ is $\phigh$. 
\end{example}
This next lemma is the technical key. Its proof may use Choice, and
this is OK, because we use it for finite $\Gamma$ (see Remark \ref{Rem:L}
below). 
\begin{lem}
\label{lem:MainLemma}(taken from \cite[The third lemma]{prohleEndomorphism}
with some adjustments) Let $r$ be a prime number or $0$, $p$ a
prime number different from $r$ and let $\left\{ p_{0},\ldots,p_{n-1}\right\} $
be a set of pairwise distinct primes, different from $p,r$. Let $F$
be a field of characteristic $r$. For $k<n$, let $V_{k}$ be some
set such that $k\neq l\Rightarrow V_{k}\cap V_{l}=\emptyset$, and
let $V=\bigcup_{k<n}V_{k}$. 

For each $v\in V$, let $T_{v}\in F\left[X\right]$ be polynomials
such that: 
\begin{itemize}
\item none of them is constant.
\item none of them is divisible by $X$. 
\item they are separable polynomials. 
\item they are pairwise relatively prime (i.e. no nontrivial common divisor).
\end{itemize}
Suppose that $K$ is an extension of $F$ generated by the set $\left\{ z_{i}\left|\, i<\omega\right.\right\} \cup\left\{ t_{i}^{v}\left|\, v\in V,i<\omega\right.\right\} $
from the algebraic closure of $F\left(z_{0}\right)$ where: 
\begin{itemize}
\item $z_{0}$ is transcendental over $F$. 
\item $\left(z_{i+1}\right)^{p}=z_{i}$ for all $i<\omega$. 
\item For $v\in V$, $t_{0}^{v}=T_{v}\left(z_{0}\right)$
\item if $v\in V_{k}$ then $\left(t_{i+1}^{v}\right)^{p_{k}}=t_{i}^{v}$. 
\end{itemize}
Then we have the following properties: 
\begin{enumerate}
\item \label{enu:transExten}$F$ is relatively algebraically closed in
$K$.
\item \label{enu:An-equivalent-definition}An equivalent definition of $K$
is the following one: Suppose $F$ is the field of fractions of an
integral domain $S$. Then $K$ is the field of fractions of the integral
domain $R/I$ (in particular $I$ is prime) where \\
 $R=S\left[Y_{i},S_{l}^{v}\left|\, i,l<\omega,v\in V\right.\right]$
(i.e. the ring generated freely by $S$ and these elements) and $I\leq R$
is the ideal generated by the equations:

\begin{enumerate}
\item $Y_{i+1}^{p}=Y_{i}$ for $i<\omega$. 
\item $S_{0}^{v}=T_{v}\left(Y_{0}\right)$ for $v\in V$. 
\item If $v\in V_{k}$, then $\left(S_{l+1}^{v}\right)^{p_{k}}=S_{l}^{v}$
for $k<n$, $l<\omega$. 
\end{enumerate}
\item \label{enu:s-high}Each $q\mbox{-high}$ element of $K$ belongs to
$F$ whenever $q$ is a prime different from $p$ and $\left\langle p_{k}\left|\, k<n\right.\right\rangle $. 
\item \label{enu:p-high}Each $\phigh$ element of $K$ is of the form $c\cdot\left(z_{i}\right)^{m}$,
where $c$ is a $\phigh$ element of $F$, $i<\omega$ and $m$ is
an integer. 
\item \label{enu:noMoreRoots}If $p'$ is a prime different from $p$ then
$z_{0}$ does not have a $p'$ root. 
\item \label{enu:Cardinality}If $V$ is finite then $\left|K\right|\leq\left|F^{\left\langle <\omega\right\rangle }\right|$.
Furthermore, the injection witnessing this is definable from the parameters
given when constructing $K$ (i.e. the function $v\mapsto T_{v}$,
etc). 
\end{enumerate}
\end{lem}
The proof may be found in Section \ref{sec:some-technical-lemmas}.

The rest of the section is devoted to proving
\begin{thm}
Let $\Gamma=\left\langle X,E,C\right\rangle $ be an $N$-colored
graph as in Lemma \ref{lem:Graph-color}. Then there exists a field
$K_{\Gamma}$ such that $\left|K_{\Gamma}\right|\leq\left|X^{\left\langle <\omega\right\rangle }\right|$
and $\Aut\left(K_{\Gamma}\right)\cong\Aut\left(\Gamma\right)$. Furthermore,
$X\subseteq K_{\Gamma}$ and $\pi\mapsto\pi\upharpoonright X$ is
an isomorphism from $\Aut\left(K_{\Gamma}\right)$ onto $\Aut\left(\Gamma\right)$.
We can choose $K_{\Gamma}$ to be of any characteristic. 
\end{thm}
So Main Theorem \ref{thm:GraphToField} immediately follows from this
and Lemma \ref{lem:Graph-color}. 

\uline{The construction of $K_{\Gamma}$}: Let $L$ be the field
$\Qq$ or $\Ff_{r}$ for some prime $r$. Let $\left\langle p_{i}\left|\, i\leq N\right.\right\rangle $
list odd prime numbers which are different then $r$, and do not divide
$r-1$ (so that in $L$ there are no $p_{i}$-roots of unity). Let
$R$ be the ring $L\left[Y_{\Gamma}\right]$ where $Y_{\Gamma}=\left\{ x_{s}^{i}\left|\, i<\omega,s\in X\right.\right\} \cup\left\{ x_{e}^{i}\left|i<\omega,e\in E\right.\right\} $%
\footnote{The $i$ s are indices not exponents! Later we will use parentheses
in order not to confuse a superscript with an exponent.%
} is an algebraically independent set. Let $I_{\Gamma}\subseteq R$
be the ideal generated by the equations:
\begin{itemize}
\item $\left(x_{s}^{i+1}\right)^{p_{0}}=x_{s}^{i}$ for all $s\in X$ and
$i<\omega$.
\item If $e=\left\{ s,t\right\} $ then $x_{e}^{0}=x_{s}^{0}+x_{t}^{0}+1$
for all $s,t\in X$ and $e\in E$.
\item If $C\left(e\right)=l$ then $\left(x_{e}^{i+1}\right)^{p_{l+1}}=x_{e}^{i}$
for all $e\in E$.
\end{itemize}
Now let $R_{\Gamma}$ be the ring $R/I_{\Gamma}$.
\begin{rem}
\label{Rem:L}$ $ 
\begin{enumerate}
\item If $\Gamma,\Gamma'$ are $N$-colored graphs, and $\Gamma\cong\Gamma'$
(and the isomorphism respects the coloring) then $R_{\Gamma}\cong R_{\Gamma'}$. 
\item Hence we may use Choice when proving properties regarding $R_{\Gamma}$
(and later $K_{\Gamma}$) when $\Gamma$ is finite because we may
assume $\Gamma\in\mathbb{L}$ (hence also $R_{\Gamma}\in\mathbb{L}$
etc). In that case we may use Lemma \ref{lem:MainLemma} even if there
is Choice in the proof.
\end{enumerate}
\end{rem}
\begin{prop}
\label{cla:I-is-prime} $I_{\Gamma}$ is prime, so we let $K_{\Gamma}$
be be the field of fractions of $R_{\Gamma}$.
\end{prop}
The proof uses the following remark (when it makes sense)
\begin{rem}
\label{rem:SituationOfLemma}If $\Gamma_{0}\subseteq\Gamma_{1}$ are
finite where $\Gamma_{i}=\left\langle X_{i},E_{i},C_{i}\right\rangle $
for $i<2$ and $X_{1}=X_{0}\cup\left\{ t\right\} $, $t\notin X_{0}$,
then the field extension $K_{\Gamma_{0}}\subseteq K_{\Gamma_{1}}$
is as in Lemma \ref{lem:MainLemma}, where
\begin{itemize}
\item $F$ is the field $K_{\Gamma_{0}}$; $r$ is its characteristic; $p$
is $p_{0}$; $\left\{ p_{0},\ldots,p_{n-1}\right\} $ is $\left\{ p_{l+1}\left|\, l<N\right.\right\} $;
$V_{k}$ is the set of edges $\left\{ t,s\right\} \in E$ of color
$k$; for $s\in X_{0}$ such that $v=\left\{ t,s\right\} \in E$,
$T_{v}$ is the polynomial $X+x_{s}^{0}+1$; $z_{i}$ is $x_{t}^{i}$
and for $v=e=\left\{ t,s\right\} $, $t_{i}^{v}$ is $x_{e}^{i}$. 
\end{itemize}
\end{rem}
\begin{proof}
(of proposition) We may assume $\Gamma$ is finite, so the proof is
by induction on $\left|X\right|$. Suppose that $\Gamma_{0}\subseteq\Gamma_{1}$
where $\Gamma_{i}=\left\langle X_{i},E_{i},C_{i}\right\rangle $ for
$i<2$ and that $X_{1}=X_{0}\cup\left\{ t\right\} $, $t\notin X_{0}$.
By induction, $I_{\Gamma_{0}}$ is prime, so $R=R_{\Gamma_{0}}$ is
an integral domain.

Let $Y_{t}=\left\{ x_{t}^{i}\left|\, i<\omega\right.\right\} \cup\left\{ x_{e}^{i}\left|\, i<\omega,t\in e\in E_{1}\right.\right\} $;
$I_{t}\subseteq R\left[Y_{t}\right]$ be the ideal generated by the
equations related to $t$ and $\left\{ e\in E_{1}\left|\, t\in e\right.\right\} $.

By Lemma \ref{lem:MainLemma}, clause (\ref{enu:An-equivalent-definition}),
$I_{t}$ is prime. 

Consider the canonical projection $\pi:L\left[Y_{\Gamma_{1}}\right]\to R\left[Y_{t}\right]$
so that $\pi\left(I_{\Gamma_{1}}\right)=I_{t}$ and $\left\langle I_{\Gamma_{0}}\right\rangle =\ker\left(\pi\right)$.
Hence, $\pi$ induces an isomorphism $L\isp{Y_{\Gamma_{1}}}/I_{\Gamma_{1}}\to R\isp{Y_{t}}/I_{t}$
and we are done since the right hand side is an integral domain. \end{proof}
\begin{defn}
\label{def:ZFC_K_Gamma}($ZFC$) Let $F$ be a field and let $p$
be a natural number. Let $S$ be a set of elements from $F$. Then
$F\left(S,p\right)$ denotes the field which is obtained by adjoining
the elements $\left\{ s\left(l\right)\left|\, s\in S,l<\omega\right.\right\} $
from the algebraic closure of $F$ where: 
\begin{itemize}
\item $s\left(0\right)=s$. 
\item $s\left(l+1\right)^{p}=s\left(l\right)$, $l<\omega$. 
\end{itemize}
\end{defn}
\begin{rem}
\label{rem:Why Choice is needed}Choice is a priori needed in this
definition because the construction implicitly assumes the existence
of an algebraic closure, and some ordering of $S$ and of the $p$-roots
of the $s\left(l\right)$s.\end{rem}
\begin{defn}
\label{def:ZFC Definition of the field}Let $K_{-1}=L\left(Y\right)\left(Y,p_{0}\right)$,
where $Y=\left\{ x_{t}^{0}\left|\, t\in X\right.\right\} $, and $L\left(Y\right)$
denotes the purely transcendental extension of $L$, and for $l<N$,
$K_{l}=K_{l-1}\left(E_{l},p_{l+1}\right)$, where $E_{l}=\left\{ x_{s}^{0}+x_{t}^{0}+1\left|\,\left\{ s,t\right\} =e\in E,C\left(e\right)=l\right.\right\} $.\end{defn}
\begin{lem}
$ $
\begin{enumerate}
\item For $\Gamma$ finite%
\footnote{The assumption that $\Gamma$ is finite is only to insure that $K_{N-1}$
is well defined, with Choice this assumption is not needed.%
}, $K_{\Gamma}$ is canonically isomorphic to $K_{N-1}$.
\item If $\Gamma_{0}\subseteq\Gamma_{1}$ then $K_{\Gamma_{0}}\subseteq K_{\Gamma_{1}}$.
\end{enumerate}
\end{lem}
\begin{proof}
(1) follows from Lemma \ref{lem:MainLemma}, (\ref{enu:An-equivalent-definition})
by induction on the size of $\Gamma$, similarly to the proof of Proposition
\ref{cla:I-is-prime}. (2) follows from (1) for finite $\Gamma$,
which is enough. 
\end{proof}
From now on, fix some $\Gamma$.
\begin{defn}
For each $Y\subseteq X$, let $\Gamma_{Y}$ be the induced subgraph
generated by $Y$ (i.e. $\Gamma_{Y}=\left\langle Y,E\upharpoonright Y\right\rangle $)
and let $R_{Y}=R_{\Gamma_{Y}}$, $K_{Y}=K_{\Gamma_{Y}}$. 
\end{defn}
Some properties of $K_{\Gamma}$:
\begin{lem}
\label{lem:SimpleLemma}For $\Gamma$ as in Lemma \ref{lem:Graph-color}, 
\begin{enumerate}
\item \label{enu:p-highFinite}For any prime $p$, if $a\in K_{Y}$ for
some $Y\subseteq X$ and is $\phigh$ in $K_{\Gamma}$ then $a$ is
already $\phigh$ in $K_{Y}$. 
\item \label{enu:X_s_indi}For each $i<\omega$, the set $\left\{ x_{s}^{i}\left|\, s\in X\right.\right\} $
is algebraically independent over $L$. 
\item \label{enu:NoNewAlgebraic}If $X_{1}\subseteq X_{2}$ then $K_{X_{1}}$
is relatively algebraically closed in $K_{X_{2}}$ (in particular
$L$ is r.a.c in $K_{\Gamma}$). 
\end{enumerate}
\end{lem}
\begin{proof}
(\ref{enu:p-highFinite}) and (\ref{enu:X_s_indi}) follows from (\ref{enu:NoNewAlgebraic}).
For (\ref{enu:NoNewAlgebraic}), we may assume $X_{1},X_{2}$ are
finite, and then it is enough to prove it for the case $X_{2}=X_{1}\cup\left\{ t\right\} ,t\notin X_{1}$.
Now use Remark \ref{rem:SituationOfLemma}, and clause (\ref{enu:transExten})
of Lemma \ref{lem:MainLemma}. 
\end{proof}
Now we shall define the isomorphism from $\Aut\left(\Gamma\right)$
to $\Aut\left(K_{\Gamma}\right)$:
\begin{prop}
\label{prop:definie_epsilon}For $\Gamma$ as in \ref{lem:Graph-color},
there is a canonical injective homomorphism $\sigma:\Aut\left(\Gamma\right)\to\Aut\left(K_{\Gamma}\right)$
defined by $\sigma\left(\varphi\right)\left(x_{t}^{i}\right)=x_{\varphi\left(t\right)}^{i}$,
and $\sigma\left(\varphi\right)\left(x_{e}^{i}\right)=x_{\varphi\left(e\right)}^{i}$,
for $\varphi\in\Aut\left(\Gamma\right)$ and all $t\in X,e\in E$. \end{prop}
\begin{proof}
$\sigma$ is well defined because of clause (\ref{enu:Preservation})
of Lemma \ref{lem:Graph-color}. $\sigma$ is obviously a homomorphism.
It is injective: If $\sigma\left(\varphi\right)=\id$, while $\varphi\left(s\right)=t\neq s$,
then $x_{s}^{0}=\sigma\left(\varphi\right)\left(x_{s}^{0}\right)=x_{t}^{0}$
--- a contradiction to clause (\ref{enu:X_s_indi}) of Lemma \ref{lem:SimpleLemma}. 
\end{proof}
Our aim is to prove that $\sigma$ is onto. We start with:
\begin{claim}
\label{cla:p-highForm}Suppose that $a\in K_{\Gamma}$ is $\phigh$,
then: 
\begin{enumerate}
\item \label{enu:p=00003D3Dp_0}If $p=p_{0}$ then $a$ can be written in
the form $\varepsilon\cdot\prod\left\{ \left(x_{s}^{n_{s}}\right)^{m_{s}}\left|\, s\in X_{0}\right.\right\} $
for some finite $X_{0}\subseteq X$ , some choice of $m_{s}\in\mathbb{Z},n_{s}<\omega$
for $s\in X_{0}$ and a $p_{0}\mbox{-high}$ element $\varepsilon\in L$.
\item \label{enu:p=00003D3Dp_{l+1}}If $p=p_{l+1}$ for some $l<N$ then
$a$ can be written in the form $\varepsilon\cdot\prod\left\{ \left(x_{e}^{n_{e}}\right)^{m_{e}}\left|\, e\in E_{0}\right.\right\} $
for some finite $E_{0}\subseteq E$ such that $C\upharpoonright E_{0}=l$,
some choice of $n_{e}<\omega,m_{e}\in\mathbb{Z}$ for $e\in E_{0}$
and a $p_{l+1}\mbox{-high}$ element $\varepsilon\in L$.
\end{enumerate}
\end{claim}
\begin{proof}
By Lemma \ref{lem:SimpleLemma}, clause (\ref{enu:p-highFinite}),
there is some $X_{0}\subseteq X$ such that $a$ is $\phigh$ in $K_{X_{0}}$.
The proof is by induction on $\left|X_{0}\right|$. The base of the
induction --- $X_{0}=\emptyset$ --- is clear. For the induction step,
we prove that if $X_{0}\subseteq X_{1}$ are finite and $X_{1}=X_{0}\cup\left\{ t\right\} $,
$t\notin X_{0}$, and the claim is true for $X_{0}$, then every $a\in K_{X_{1}}$
which is $\phigh$ has the desired form.

For clause (\ref{enu:p=00003D3Dp_0}), Remark \ref{rem:SituationOfLemma}
implies that we can use Lemma \ref{lem:MainLemma}, clause (\ref{enu:p-high}).

For (\ref{enu:p=00003D3Dp_{l+1}}), we shall use the assumption on
the coloring. 
\begin{caseenv}
\item There is no edge $e_{0}\ni t$ in $\Gamma_{X_{1}}$ such that $C\left(e_{0}\right)=l$.
In that case, we use clause (\ref{enu:s-high}) of Lemma \ref{lem:MainLemma},
and conclude that $a\in K_{X_{0}}$. 
\item There is an edge $e_{0}\ni t$ in $\Gamma_{X_{1}}$ with $C\left(e_{0}\right)=l$,
but only one such edge. If $e_{0}=\left\{ s,t\right\} ,s\in X_{0}$
then $x_{e_{0}}^{0}=x_{s}^{0}+x_{t}^{0}+1\in K_{X_{1}}$ is transcendental
over $K_{X_{0}}$ (because $x_{t}^{0}$ is). In addition $x_{t}^{0}=x_{e_{0}}^{0}-x_{s}^{0}-1$
and for all vertices $r\in X_{0}$ such that $e_{r}=\left\{ t,r\right\} $
is an edge (of some other color), $x_{e_{r}}^{0}=x_{e_{0}}^{0}-x_{s}^{0}+x_{r}^{0}$.
The polynomials $X-x_{s}^{0}-1$, $X-x_{s}^{0}+x_{r}^{0}$ satisfy
the conditions of Lemma \ref{lem:MainLemma}, and so, by clause (\ref{enu:p-high}),
$a$ is of the form $\left(x_{e_{0}}^{i}\right)^{m}\cdot c$ for $c$
which is $p_{l+1}\mbox{-high}$ in $K_{\Gamma_{0}}$ and we are done
(we do not use the lemma in the same way as in Remark \ref{rem:SituationOfLemma}
--- here $z_{0}$ is played by $x_{e_{0}}^{0}$, but it is the same
idea). 
\item There is more than one edge $e_{0}\ni t$ in $\Gamma_{X_{1}}$ with
color $l$. Then $t$ is the center of a star in the subgraph of $\Gamma_{1}$
induced by that color. Assume that $s_{1},\ldots,s_{k}\in X_{0}$
list the vertices such that $C\left(s_{i},t\right)=l$, ($k\geq2$).
Let $X^{-}=X_{0}\backslash\left\{ s_{1},\ldots,s_{k}\right\} $, and
$X'=X^{-}\cup\left\{ t\right\} $. Note that $\left|X'\right|<\left|X_{1}\right|$,
so by the induction hypothesis, the claim is true for $K_{X'}$. $\Gamma_{X_{1}}$
is built from $\Gamma_{X'}$ by adding $s_{1},\ldots,s_{k}$ and in
each step we are in the previous case (because $t$ was the center
of a star), so we are done. 
\end{caseenv}
\end{proof}
\begin{lem}
\label{lem:x_0DoesntHaveRoots}For all $s\in X$, $x_{s}^{0}$ does
not have a $p'$ root for $p'$ a prime different from $p_{0}$. \end{lem}
\begin{proof}
Again, it is enough to prove this finite $X_{0}\subseteq X$, and
the proof is by induction on $\left|X_{0}\right|$, and follows from
clause (\ref{enu:noMoreRoots}) of Lemma \ref{lem:MainLemma}.
\end{proof}
This is the main proposition: 
\begin{prop}
\label{pro:MainClaim1} Assume $\varphi\in\Aut\left(K_{\Gamma}\right)$
and that $\left\{ s_{0},t_{0}\right\} \in E$ of color $l$. Then
there is an edge $\left\{ s_{1},t_{1}\right\} \in E$ of the same
color such that $\varphi\left(x_{s_{0}}^{0}\right)=x_{s_{1}}^{0}$
and $\varphi\left(x_{t_{0}}^{0}\right)=x_{t_{1}}^{0}$. \end{prop}
\begin{proof}
Let $f_{1}=\varphi\left(x_{s_{0}}^{0}\right)$, $f_{2}=\varphi\left(x_{t_{0}}^{0}\right)$,
$f=\varphi\left(x_{s_{0}}^{0}+x_{t_{0}}^{0}+1\right)=f_{1}+f_{2}+1$.
From Claim \ref{cla:p-highForm} it follows that
\begin{itemize}
\item $f_{1}=\varepsilon_{1}\cdot\prod\left\{ \left(x_{s}^{i_{s}}\right)^{m_{s}}\left|\, s\in X_{0}\right.\right\} $,
$f_{2}=\varepsilon_{2}\cdot\prod\left\{ \left(x_{t}^{i_{s}}\right)^{m_{t}}\left|\, t\in Y_{0}\right.\right\} $
and \\
$f=\varepsilon_{3}\cdot\prod\left\{ \left(x_{e}^{i_{e}}\right)^{m_{e}}\left|\, e\in E_{0}\right.\right\} $,
\end{itemize}
where $X_{0},Y_{0}\subseteq X$ and $E_{0}\subseteq E$ are finite
nonempty; $i_{s}<\omega$, $m_{s}\in\mathbb{Z}$ for $s\in X_{0}$;
$i_{t}<\omega$, $m_{t}\in\mathbb{Z}$ for $t\in Y_{0}$; and $E_{0}$
is homogeneous of color $l$ and $i_{e}<\omega$, $m_{e}\in\mathbb{Z}$
for $e\in E_{0}$. Let $p=p_{l+1}$, so $f$ is $\phigh$.

We can assume that unless $i_{s}=0$, $p_{0}\nmid m_{s}$ for $s\in X_{0}\cup Y_{0}$,
and that unless $i_{e}=0$, $p\nmid m_{e}$ for $e\in E_{0}$. 

Raising the equation $f_{1}+f_{2}+1=f$ by $p^{k}$ where $k=\max\left\{ i_{e}\left|\, e\in E_{0}\right.\right\} $,
we have an equation of the form
\[
\left(\varepsilon_{1}\prod\left(x_{s}^{i_{s}}\right)^{m_{s}}+\varepsilon_{2}\prod\left(x_{t}^{i_{t}}\right)^{m_{t}}+1\right)^{p^{k}}=\varepsilon_{3}^{p^{k}}\prod\left(x_{r}^{0}+x_{w}^{0}+1\right)^{p^{k-i_{\left\{ r,w\right\} }}m_{\left\{ r,w\right\} }}.
\]
Let $i=\max\left\{ i_{t}\left|\, t\in X_{0}\cup Y_{0}\right.\right\} $.
We can replace $x_{t}^{i_{t}}$ by $\left(x_{t}^{i}\right)^{p_{0}^{i-i_{t}}}$
and the same for $x_{s}^{i_{s}}$. Also replace $x_{r}^{0}$ by $\left(x_{r}^{i}\right)^{p_{0}^{i}}$
and the same for $x_{w}^{0}$. For $t\in T:=X_{0}\cup Y_{0}\cup\bigcup E_{0}$,
let $y_{t}=x_{t}^{i}$, then we get 
\[
\left(\varepsilon_{1}\prod\left(y_{s}\right)^{p_{0}^{i-i_{s}}m_{s}}+\varepsilon_{2}\prod\left(y_{t}\right)^{p_{0}^{i-i_{t}}m_{t}}+1\right)^{p^{k}}=\varepsilon_{3}^{p^{k}}\prod\left(\left(y_{r}\right)^{p_{0}^{i}}+\left(y_{w}\right)^{p_{0}^{i}}+1\right)^{p^{k-i_{\left\{ r,w\right\} }}m_{\left\{ r,w\right\} }}.
\]
By Lemma \ref{lem:SimpleLemma}, these elements are algebraically
independent so this is an equation in the field of rational functions
$L\left(y_{t}\left|\, t\in T\right.\right)$.

The next step is to see that the exponents ($m_{t}$ and $m_{\left\{ r,w\right\} }$)
are non-negative. For that we use valuations. 

Recall that for any field, $F$ and any irreducible $g\in F\left[X\right]$
there is a unique discrete (i.e. with value group $\mathbb{Z}$) valuation
on the field of rational functions $F\left(t\right)$ defined by $v\left(g\left(t\right)\right)=1$,
$v\upharpoonright F^{\times}=0$. In this case, $v\upharpoonright F\left[t\right]\geq0$
and $v\left(m\left(t\right)\right)>0$ iff $g|m$ for $m\in F\left[X\right]$.
This is the $g$-adic valuation.

Suppose $m_{t_{0}}$ is negative for some $t\in X_{0}\cup Y_{0}$.
Consider the discrete valuation $v$ on the field $L\left(y_{t}\left|\, t\in T\right.\right)$
defined by $v\left(y_{t_{0}}\right)=1$, $v\upharpoonright L\left(y_{t}\left|\, t\neq t_{0}\right.\right)^{\times}=0$.
Then on the left hand side we get $v\left(LHS\right)<0$ while on
the right hand side, $v\left(RHS\right)=0$ --- contradiction.

Suppose $m_{\left\{ r,w\right\} }<0$ for some $\left\{ r,w\right\} \in E_{0}$.
Consider the valuation $v$ on the field $L\left(y_{t}\left|\, t\in T\right.\right)$
defined by $v\left(g\left(y_{r}\right)\right)=1$, $v\upharpoonright L\left(y_{t}\left|\, t\neq r\right.\right)^{\times}=0$
where $g$ is any irreducible polynomial dividing $X^{p_{0}^{i}}+\left(y_{w}\right)^{p_{0}^{i}}+1$.
So $v\left(\left(y_{r}\right)^{p_{0}^{i}}+\left(y_{w}\right)^{p_{0}^{i}}+1\right)>0$,
while $g$ does not divide $\left(X^{p_{0}^{i}}+\left(y_{w'}\right)^{p_{0}^{i}}+1\right)$
for $w\neq w$' (they relatively prime) so $v\left(RHS\right)<0$.
On the other hand, since $v\left(y_{r}\right)=0$, $v\left(RHS\right)\geq0$
--- contradiction.

Hence we can consider this equation as one in the polynomial ring
$L\left[y_{t}\left|\, t\in T\right.\right]$. Moreover, since these
elements are algebraically independent, each one appearing in the
left hand side must appear in the right hand side and vice versa,
i.e. $T=X_{0}\cup Y_{0}=\bigcup E_{0}$. 

By examining the free factor, $\varepsilon_{3}^{p^{k}}=1$.

By substituting $y_{r}$ and $y_{w}$ with $0$ for some $r,w$ ,
we can show that $E_{0}=\left\{ \left\{ r,w\right\} \right\} $ (so
$k=i_{\left\{ r,w\right\} }$) and that there are no mixed monomials
in the left hand side, i.e. we get an equation of the form
\[
\left(\varepsilon_{1}\left(y_{r}\right)^{p_{0}^{i-i_{r}}m_{r}}+\varepsilon_{2}\left(y_{w}\right)^{p_{0}^{i-i_{w}}m_{w}}+1\right)^{p^{k}}=\left(\left(y_{r}\right)^{p_{0}^{i}}+\left(y_{w}\right)^{p_{0}^{i}}+1\right)^{m_{\left\{ r,w\right\} }}.
\]
Suppose $i=i_{r}$ and $i\neq0$, then $p_{0}\nmid m_{r}$, by examining
the degree of $y_{r}$, we get a contradiction, so $i=0$ and by choice
of $i$, $i_{w}=0$ as well. In the same way we can deduce that $k=0$.
From this it follows that $\varepsilon_{1}=\varepsilon_{2}=1$ and
$m_{r}=m_{w}$. So we have
\[
f_{1}+f_{2}+1=\left(x_{r}^{0}\right)^{m_{r}}+\left(x_{w}^{0}\right)^{m_{w}}+1=\left(x_{r}^{0}+x_{w}^{0}+1\right)^{m_{\left\{ r,w\right\} }}=f.
\]
So $\left\{ r,w\right\} $ is an edge of color $l$, $m:=m_{\left\{ r,w\right\} }=m_{w}=m_{r}$,
and $m=1$ or a power of $r$ (the characteristic). 

So finally we have that $\varphi\left(x_{t_{0}}^{0}\right)$ is a
power of $m$ which is a power of $r$. This implies that $x_{t_{0}}^{0}$
itself has an $m$-root. But if $m>1$, this is a contradiction, because
$x_{t_{0}}^{0}$ has no $r$-roots by Lemma \ref{lem:x_0DoesntHaveRoots}.

This concludes the proof of the proposition.\end{proof}
\begin{cor}
\label{cor:MainClaim2}The map $\sigma:\Aut\left(\Gamma\right)\to\Aut\left(K_{\Gamma}\right)$
is a bijection.\end{cor}
\begin{proof}
Recall that all that is left is to show that $\sigma$ is onto (by
Proposition \ref{prop:definie_epsilon}). 

Let $\varphi\in\Aut\left(K_{\Gamma}\right)$. Let $t\in X$ and suppose
$\left\{ t,t_{0}\right\} \in E$. By Proposition \ref{pro:MainClaim1},
$\varphi\left(x_{t}^{0}\right)=x_{t'}^{0}$ for the some $t'\in X$.
Since the graph $\Gamma$ is connected, we can define $\varepsilon\in\Aut\left(\Gamma\right)$
by $\varepsilon\left(t\right)=t'$ (note that $t'$ does not depend
on the choice of $t_{0}$). Proposition \ref{pro:MainClaim1} implies
that $\varepsilon$ is indeed an automorphism. 

Since there are no $p_{i}$-roots of unity in $L$ for all the primes
we chose, it follows then that $\varphi\left(x_{t}^{i}\right)=x_{\varepsilon\left(t\right)}^{i}$
and that $\varphi\left(x_{e}^{i}\right)=x_{\varepsilon\left(e\right)}^{i}$,
and hence $\varphi=\sigma\left(\varepsilon\right)$.
\end{proof}
We still have to prove that $\left|K_{\Gamma}\right|\leq\left|X^{\left\langle <\omega\right\rangle }\right|$. 
\begin{lem}
\label{lem:Intersection} If $X_{i}\subseteq X$ ($i=1,2$) are two
subsets of the vertices set then $K_{X_{1}}\cap K_{X_{2}}=K_{X_{1}\cap X_{2}}$. \end{lem}
\begin{proof}
We may assume that $X_{1},X_{2}$ are finite. Assume $x\in K_{X_{1}}\cap K_{X_{2}}$
and that $\left|X_{1}\right|$ is minimal with respect to $x\in K_{X_{1}}$.
If $X_{1}\subseteq X_{2}$ then we are done. If not, let $t\in X_{1}\backslash X_{2}$
be some vertex, and let $X'=X_{1}\backslash\left\{ t\right\} $. So
$x\notin K_{X'}$, and $x$ is transcendental over $K_{X'}$ while
$x_{t}^{0}$ is algebraic over $K_{X'}\left(x\right)$. Let $X'_{2}=X'\cup X_{2}$,
$X_{3}=X'_{2}\cup\left\{ t\right\} $. We have $x\in K_{X_{2}}\subseteq K_{X'_{2}}$,
and $x_{t}^{0}\in K_{X_{3}}$ is transcendental over $K_{X'_{2}}$.
This is a contradiction, because $x_{t}^{0}$ is algebraic over $K_{X'}\left(x\right)\subseteq K_{X'_{2}}$.
Hence there is no such $t$ i.e. $X_{1}\subseteq X_{2}$. 
\end{proof}
And now it is easy to define an injective map $\Psi:K_{\Gamma}\to X^{\left\langle <\omega\right\rangle }$.
Define by induction on $n$ injective function $\Psi_{Y}:K_{Y}\to X^{\left\langle <\omega\right\rangle }$
for $\left|Y\right|\leq n$ such that $Y_{1}\subseteq Y_{2}$ implies
$\Psi_{Y_{1}}\subseteq\Psi_{Y_{2}}$. This is enough, since by the
lemma above, $\bigcup\left\{ \Psi_{Y}\left|\, Y\subseteq X,\left|Y\right|<\omega\right.\right\} $
is an injection from $K_{\Gamma}$ to $X^{\left\langle <\omega\right\rangle }$. 

For the construction of $\Psi_{Y}:K_{Y}\to X^{\left\langle <\omega\right\rangle }$,
the idea is that given $x\in K_{Y}$ such that $x\notin K_{Y'}$ for
any $Y'\subsetneq Y$ we can code $x$ using the set $Y$ and the
set of codes that Lemma \ref{lem:MainLemma}, clause (6) gives us
for any choice of $Y'\subsetneq Y$ of size $\left|Y\right|-1$. 

This (and Lemma \ref{lem:MainLemma}, clause (6)) was the reason we
chose $X^{\left\langle <\omega\right\rangle }$ and not $X^{<\omega}$:
in order to code $x\in K_{\Gamma}$, we need first to code the minimal
set $Y$ such that $x\in K_{Y}$, and then $x$ can be coded in $\left|Y\right|$
different ways, depending on the choice of $\left|Y'\right|$ as above.
However, there is no well ordering of $Y$, so we have no way of ordering
these codes. For instance, the code of $x_{t}^{0}+x_{s}^{0}$ for
$s,t\in X$, should be $\left\{ \left\langle s\right\rangle ,\left\langle t\right\rangle ,\ldots\right\} $.

\section{\label{sec:some-technical-lemmas}some technical lemmas on fields}

This section is devoted to technical lemmas concerning fields. We
may use Choice here --- see Remark \ref{Rem:L}.

First, some simple and known facts: 
\begin{fact}
\label{lem:Abel-Theorem}(Abel's Theorem) Suppose that $p$ is prime
and $K$ is a field. Then the polynomial $X^{p}-a$ is irreducible
iff $a$ does not have a $p$-th root in $K$. \end{fact}
\begin{lem}
\label{lem:Kummer}Let $n$ be a positive integer and let $K$ be
a field of characteristic $r$, where $r=0$ or $r\nmid n$, which
contains a primitive $n$-th root of unity. Let $0\neq a\in K$ and
suppose $z$ is a root of the equation $X^{n}=a$. If $b\in K\left(z\right)$
satisfies $b^{n}\in K$, then $b=c\cdot z^{k}$ for some $0\leq k<n$
and $c\in K$. \end{lem}
\begin{proof}
This is an easy consequence of Kummer Theory. See \cite[VI.8, Theorem 8.2]{Lang}.\end{proof}
\begin{lem}
\label{lem:dimOfExtension}Let $K$ be a field containing all roots
of unity. Assume that $t$ solves the equation $X^{p}=a$ for some
$a\in K$ and prime $p$, and $L=K\left(t\right)$. Then if $b\in L$
satisfies $b^{q^{m}}\in K$ for some prime $q\neq p,m<\omega$, then
$b\in K$. \end{lem}
\begin{proof}
By Abel's theorem, and since $K$ contains all $p$ and $q$ roots
of unity, $\left[L:K\right]=p$ or $\left[L:K\right]=1$, while $\left[K\left(b\right):K\right]$
is a power of $q$, so it must be $1$. \end{proof}
\begin{lem}
\label{lem:PurelyTrans}Assume $K$ and $L$ are fields such that: 
\begin{enumerate}
\item \label{enu:KTransL}$K\supseteq L$ and is a $L$ is relatively algebraically
closed in $K$.
\item $K$ is a finite algebraic extension of the simple transcendental
extension $L\left(y\right)$. 
\end{enumerate}
Then if $p$ is a prime and $x\in K$ is $\phigh$, then $x\in L$. \end{lem}
\begin{proof}
(This proof is taken from \cite{prohleEndomorphism}). Assume $x\in K\backslash L$.
Then $y$ is algebraic over $L\left(x\right)$ (by (\ref{enu:KTransL})).
Denote by $x_{m}$ for $m<\omega$ the $p^{m}$-th root of $x$ given
in Definition \ref{def:p-high}. Then we have $L\left(x\right)\subseteq L\left(x_{1}\right)\subseteq L\left(x_{2}\right)\subseteq\ldots\subseteq K$.
As $K/L\left(x\right)$ is algebraic of finite degree, $L\left(x_{l}\right)=L\left(x_{l+1}\right)$
for some $l$, so $x_{l+1}\in L\left(x_{l}\right)$ --- the transcendental
element $x_{l}$ has a $p$ root in $L\left(x_{l}\right)$ --- this
is a contradiction. 
\end{proof}
Another easy fact:
\begin{fact}
\label{cla:SillyClaim}Let $R$ be an integral domain, $F$ its field
of fractions. $\alpha_{1},\ldots,\alpha_{n}$ elements algebraic over
$F$ such that: 
\begin{itemize}
\item The minimal \emph{monic} polynomial of $\alpha_{1}$ over $F$, $m_{1}\left(X_{1}\right)$,
belongs to $R\left[X_{1}\right]$. 
\item The minimal \emph{monic} polynomial of $\alpha_{2}$ over $F\left(\alpha_{1}\right)$,
$m_{2}\left(\alpha_{1},X_{2}\right)$, belongs to $R\left[\alpha_{1},X_{2}\right]$. 
\item And so on. 
\end{itemize}
Then $R\left[\alpha_{1},\ldots,\alpha_{n}\right]=R\left[X_{1},\ldots,X_{n}\right]/\left(m_{1},m_{2},\ldots,m_{n}\right)$.
In particular, $\left(m_{1},\ldots,m_{n}\right)$ is prime. 
\end{fact}
And here is the main technical lemma:
\begin{lem}
\label{lem:MainLemmaWithPf}(an expanded version of \cite[The third lemma]{prohleEndomorphism})
Let $r$ be a prime number or $0$, $p$ a prime number different
from $r$ and let $\left\{ p_{0},\ldots,p_{n-1}\right\} $ be a set
of pairwise distinct primes, different from $p,r$. Let $F$ be a
field of characteristic $r$ which contains all roots of unity. For
$k<n$, let $V_{k}$ be some set such that $k\neq l\Rightarrow V_{k}\cap V_{l}=\emptyset$,
and let $V=\bigcup_{k<n}V_{k}$.

For each $v\in V$, let $T_{v}\in F\left[X\right]$ be polynomials
such that: 
\begin{itemize}
\item none of them is constant.
\item none of them is divisible by $X$. 
\item they are separable polynomials. 
\item they are pairwise relatively prime (i.e. no nontrivial common divisor).
\end{itemize}
Suppose that $K=K_{\left\langle T_{v}\left|\, v\in V\right.\right\rangle }$
is an extension of $F$ generated by the set \\
 $\left\{ z_{i}\left|\, i<\omega\right.\right\} \cup\left\{ t_{i}^{v}\left|\, v\in V,i<\omega\right.\right\} $
from the algebraic closure of $F\left(z_{0}\right)$ where: 
\begin{itemize}
\item $z_{0}$ is transcendental over $F$. 
\item $\left(z_{i+1}\right)^{p}=z_{i}$ for all $i<\omega$. 
\item $t_{0}^{v}=T_{v}\left(z_{0}\right)$. 
\item if $v\in V_{k}$ then $\left(t_{i+1}^{v}\right)^{p_{k}}=t_{i}^{v}$. 
\end{itemize}
Let $1\leq j\leq\omega$, and $\rho:V\to\left(\omega+1\backslash\left\{ 0\right\} \right)$.
Denote the subfield \\
 $F\left(z_{i},t_{l_{v}}^{v}\left|\, i<j,l_{v}<\rho\left(v\right),v\in V\right.\right)$
of $K$ by $F\left(j,\rho\right)=F\left(j,\rho\right)_{\left\langle T_{v}\left|\, v\in V\right.\right\rangle }$. 

Then we have the following properties: 
\begin{enumerate}
\item \label{enu:x^p-z_iPf}The polynomial $X^{p}-z_{j-1}$ is irreducible
over $F\left(j,\rho\right)$ for every $\rho$ and $1\leq j$. 
\item \label{enu:x^q-t_iPf}If $w\in V_{k}$ then the polynomial $X^{p_{k}}-t_{\rho\left(w\right)-1}^{w}$
is irreducible over $F\left(j,\rho\right)$ for all $\rho,j$ such
that $\rho\left(w\right)<\omega$. 
\item \label{enu:NormalFormPf}If $k<n$ and the $\left(p_{k}\right)^{m}$-th
power ($1\leq m$) of an element of $F\left(j,\rho\right)$ belongs
to the subfield $F\left(z_{l}\right)$ where $l<j\leq\omega$ then
this element can be written in the form 
\[
c\cdot\frac{f\left(z_{i}\right)}{g\left(z_{i}\right)}\prod_{v\in W_{k}}\left(t_{r_{v}}^{v}\right)^{l_{v}}
\]
 for some $c\in F$, $f$ and $g$ are relatively prime monic polynomials
over $F$, $i\leq l$, $W_{k}$ is a finite subset of $V_{k}$ where
$v\in W_{k}\Rightarrow1\leq r_{v}<\rho\left(v\right)$, $r_{v}\leq m$,
and $0<l_{v}<\left(p_{k}\right)^{r_{v}}$. 
\item \label{enu:transExtenPf}$F$ is relatively algebraically closed in
$K$.
\item \label{enu:An-equivalent-definitionPf}An equivalent definition of
$K$ ($F\left(j,\rho\right)$) is the following one: Suppose $F$
is the field of fractions of an integral domain $S$. Then $K$ ($F\left(j,\rho\right)$)
is the field of fractions of the integral domain $R/I$ where \\
 $R=S\left[Y_{i},S_{l}^{v}\left|\, i,l<\omega,v\in V\right.\right]$
($R=S\left[Y_{i},S_{l_{v}}^{v}\left|\, i<j,l_{v}<\rho\left(v\right),v\in V\right.\right]$)
(i.e. this is a polynomial ring) and $I\leq R$ is the ideal generated
by the equations:

\begin{enumerate}
\item $Y_{i+1}^{p}=Y_{i}$ for $i<\omega$ ($i<j$) 
\item $Y_{0}^{v}=T_{v}\left(Y_{0}\right)$ for $v\in V$. 
\item If $v\in V_{k}$ for $k<n$, $\left(S_{l+1}^{v}\right)^{p_{k}}=S_{l}^{v}$
for $l<\omega$ ($l<\rho\left(v\right)$). 
\end{enumerate}
\item \label{enu:s-highPf}Each $q\mbox{-high}$ element of $K$ belongs
to $F$ whenever $q$ is a prime different from $p$ and $\left\langle p_{k}\left|\, k<n\right.\right\rangle $. 
\item \label{enu:p-highPf}Each $\phigh$ element of $K$ is of the form
$c\cdot\left(z_{i}\right)^{m}$, where $c$ is a $\phigh$ element
of $F$, $i<\omega$ and $m$ is an integer. 
\item \label{enu:noMoreRootsPf}If $p'$ is a prime different from $p$
then $z_{0}$ does not have a $p'$ root. 
\item \label{enu:CardinalityPf}If $V$ is finite then $\left|K\right|\leq\left|F^{\left\langle <\omega\right\rangle }\right|$.
Furthermore, the injection witnessing this is definable from the parameters
given when constructing $K$ (i.e. the function $v\mapsto T_{v}$,
etc). 
\item \label{enu:FAnyFieldPf}Clauses (\ref{enu:x^p-z_iPf})--(\ref{enu:CardinalityPf})
except clause (\ref{enu:NormalFormPf}) are true for \emph{any} field
$F$ of characteristic $r$. 
\end{enumerate}
\end{lem}
\begin{proof}
This proof is an adaptation of \cite[Third Lemma]{prohleEndomorphism}.
There it is dealt with just adding one $q$ root to $T_{v}$, while
we deal with infinite such roots. The difference between the proofs
is not large.

Let us assume that $n=1$. i.e. there is only one prime different
from $p,r$ and denote it by $q$. The proof is the essentially the
same if $n>1$, but involves more indices, and after reading the proof
for this case, the general case should be easy. Throughout the proof,
let $\rho:V\to\left(\omega+1\backslash\left\{ 0\right\} \right)$,
$\supp\left(\rho\right)=\left\{ v\left|\,\rho\left(v\right)\neq1\right.\right\} $
and $\left|\rho\right|=\sum\left\{ \rho\left(v\right)-1\left|v\in\supp\left(\rho\right),\rho\left(v\right)<\omega\right.\right\} $.
When we say that $\supp\left(\rho\right)$ is finite, we also mean
that $\left|\rho\right|$ is finite, and $\rho\left[V\right]\subseteq\omega$.

First let us note that it is enough to prove (\ref{enu:x^p-z_iPf}),
(\ref{enu:x^q-t_iPf}) and (\ref{enu:NormalFormPf}) for finite $\supp\left(\rho\right)$
and $j$. In addition we may assume for these clauses that $j=1$: 

Suppose $i<\omega$, and let $S_{v}^{i}\in F\left[X\right]$ be $S_{v}^{i}\left(X\right)=T_{v}\left(X^{p^{i}}\right)$
for $v\in V$. Then $\left\{ S_{v}^{i}\left|\, v\in V\right.\right\} $
satisfy the conditions of the lemma. 

Note that for finite $j$ and $\rho$ with finite $\supp\left(\rho\right)$,
$F\left(j,\rho\right)_{\left\langle T_{v}\left|\, v\in V\right.\right\rangle }\cong F\left(1,\rho\right)_{\left\langle S_{v}^{j-1}\left|\, v\in V\right.\right\rangle }$
(taking $z_{j-1}$ to $z_{0}$ and $T_{v}$ to $S_{v}^{j-1}$). Hence
if we know (\ref{enu:x^p-z_iPf}) and (\ref{enu:x^q-t_iPf}) for the
case $j=1$, then they are true for any $j$. Regarding (\ref{enu:NormalFormPf}),
we note that by Lemma \ref{lem:dimOfExtension}, if an element $x\in F\left(i+1,\rho\right)$
satisfies $x^{q^{m}}\in F\left(i,\rho\right)$ then $x$ belongs to
$F\left(i,\rho\right)$. Hence we may assume $j=l+1$, and after applying
the isomorphism above --- $j=1$.

So let us begin:

First we prove (\ref{enu:x^q-t_iPf}) and (\ref{enu:NormalFormPf}).
We prove this by induction on $\left|\rho\right|$. For $\left|\rho\right|=0$,
$F\left(1,\rho\right)=F\left(z_{0}\right)$ is just the quotient field
of the polynomial ring $F\left[z_{0}\right]$, therefore (\ref{enu:NormalFormPf})
is true in that case. Now we prove that if (\ref{enu:NormalFormPf})
is true for $\rho$ then (\ref{enu:x^q-t_iPf}) is true as well. So,
in order to prove (\ref{enu:x^q-t_iPf}), it is enough, by Abel's
Theorem (Lemma \ref{lem:Abel-Theorem}), to prove that $t_{\rho\left(w\right)}^{w}\notin F\left(1,\rho\right)$.
If this is not the case, then, by (\ref{enu:NormalFormPf}), we get
an equation of the form: 
\[
g\left(z_{0}\right)\cdot t_{\rho\left(w\right)}^{w}=cf\left(z_{0}\right)\prod_{v\in W}\left(t_{r_{v}}^{v}\right)^{l_{v}}
\]
 For some finite $W\subseteq V$, $1\leq r_{v}\leq\rho\left(w\right)$,
$0<l_{v}<q^{r_{v}}$, $r_{v}<\rho\left(v\right)$ for $v\in W$. After
raising both sides of the equation to the power of $q$, $\rho\left(w\right)$
times, we get an equation of the form
\[
g^{q^{\rho\left(w\right)}}\cdot T_{w}=c^{q^{\rho\left(w\right)}}f^{q^{\rho\left(w\right)}}\prod_{v\in W}\left(T_{v}^{q^{\rho\left(w\right)-r_{v}}}\right)^{l_{v}}
\]
 By the conditions on the polynomials $T_{v}$, $g=f=1$, and we get
a contradiction (because we get $W=\left\{ w\right\} $ and $r_{w}=\rho\left(w\right)$).

Now the induction step for (\ref{enu:NormalFormPf}). Suppose $b\in F\left(1,\rho\right)$
and $b^{q^{m}}\in F\left(z_{0}\right)$ (assume $m>0$), and let $v\in V$
be such that $\rho\left(v\right)>1$.

Define $\rho'$ by: 
\begin{itemize}
\item $\rho'\left(v\right)=\rho\left(v\right)-m$ (unless $\rho\left(v\right)-m<1$
and then $\rho'\left(v\right)=1$). 
\item $\rho'\left(w\right)=\rho\left(w\right)$ for $w\neq v$. 
\end{itemize}
So $F\left(1,\rho'\right)\left(t_{\rho\left(v\right)-1}^{v}\right)=F\left(1,\rho\right)$.
Now, $b\in F\left(1,\rho'\right)\left(t_{\rho\left(v\right)-1}^{v}\right)$,
$b^{q^{m}}\in F\left(z_{0}\right)\subseteq L:=F\left(1,\rho'\right)$
and $\left(t_{\rho\left(v\right)-1}^{v}\right)^{q^{m}}\in L$. By
Kummer Theory (Lemma \ref{lem:Kummer}), we have $b=c\cdot\left(t_{\rho\left(v\right)-1}^{v}\right)^{l}$
for some $c\in L$, $0\leq l<q^{m}$. Note that if $q\mid l$ then
we are done by induction, so assume $q\nmid l$.

If $\rho\left(v\right)-1\leq m$ the we are done: it follows that
$c^{q^{m}}\in F\left(z_{0}\right)$ and by the induction hypothesis
we know $c$ can be represented in the right form (and $t_{v}$ does
not appear there, as $\rho'\left(v\right)=1$). So assume $\rho\left(v\right)-1>m$
.

Surely, $c^{q^{\rho\left(v\right)-1}}\in F\left(z_{0}\right)$, so
by the induction hypothesis (recall $c\in L$), $c$ can be written
in the form 
\[
d\cdot\frac{f\left(z_{0}\right)}{g\left(z_{0}\right)}\prod_{u\in W}\left(t_{r_{u}}^{u}\right)^{l_{u}}
\]
where $d\in F$, $W\subseteq V$ and finite, and $1\leq r_{u}<\rho'\left(u\right),r_{u}\leq\rho\left(v\right)-1$
for $u\in W$ (hence, of course, $W\subseteq\supp\left(\rho'\right)$).
By this representation of $c$, $c^{q^{m}}\in F\left(1,\rho''\right)$
where $\rho''\left(w\right)=\rho'\left(w\right)-m$ for all $w\in V$
(and again, if $\rho'\left(w\right)-m<1$, $\rho''\left(w\right)=1$).
Since $b^{q^{m}}\in F\left(z_{0}\right)$, $\left(t_{\rho\left(v\right)-1}^{v}\right)^{l\cdot q^{m}}=\left(t_{\rho\left(v\right)-m-1}^{v}\right)^{l}\in F\left(1,\rho''\right)$.
Since $q\nmid l$, $\left(l,q^{\rho\left(v\right)-m-1}\right)=1$
so $t_{\rho\left(v\right)-m-1}^{v}\in F\left(1,\rho''\right)$. But
$\rho''\left(v\right)\leq\rho'\left(v\right)-1=\rho\left(v\right)-m-1$,
and we get a contradiction to (\ref{enu:x^q-t_iPf}) (because it follows
that $t_{\rho''\left(v\right)}^{v}\in F\left(1,\rho''\right)$). 

So (\ref{enu:x^q-t_iPf}) and (\ref{enu:NormalFormPf}) are proven.

Now we prove (\ref{enu:x^p-z_iPf}) for $j=1$ and finite $\left|\rho\right|$
by induction on $\left|\rho\right|$. By Abel's Theorem it is enough
to prove that $z_{1}\notin F\left(1,\rho\right)$. For $\left|\rho\right|=0$,
it is clear. The induction step follows from \ref{lem:dimOfExtension}.

Next we prove (\ref{enu:transExtenPf}). Again we assume that $\left|\rho\right|$
is finite. Let $x$ be an algebraic element of $K$ over $F$. Let
$L=F\left(x\right)$. The element $z_{0}$ is transcendental over
$L$, since $x$ is algebraic. All the other conditions of the lemma
are also satisfied with respect to $L$ instead of $F$. Let $v\in V$,
and $\rho^{+}:V\to\omega+1$ defined by $\rho^{+}\left(v\right)=\rho\left(v\right)+1$
and for $w\neq v$, $\rho^{+}\left(w\right)=\rho\left(w\right)$.
Suppose $x\in F\left(\omega,\rho^{+}\right)\backslash F\left(\omega,\rho\right)$.
Then $F\left(\omega,\rho\right)\left(x\right)=F\left(\omega,\rho^{+}\right)$
(because $\left[F\left(\omega,\rho^{+}\right):F\left(\omega,\rho\right)\right]=q$
--- a prime --- by (\ref{enu:x^q-t_iPf})) and in particular $t_{\rho\left(v\right)}^{v}\in F\left(\omega,\rho\right)\left(x\right)\subseteq L\left(\omega,\rho\right)$
--- a contradiction. So inductively we get $x\in F\left(\omega,1\right)$
(where $1$ is the constant sequence). Hence, $x\in F\left(z_{i}\right)$,
so $x\in F$.

Next we prove (\ref{enu:An-equivalent-definitionPf}). Denote by $S\left(j,\rho\right)$
and $I\left(j,\rho\right)$ the ring $R$ and ideal $I$ mentioned
in (\ref{enu:An-equivalent-definitionPf}). We shall show that $S\left(j,\rho\right)/I\left(j,\rho\right)$
is naturally embedded in $K$. It is enough as the field of fractions
contains all of $F\left(j,\rho\right)$'s generators.

It is enough to show this for finite $j,\left|\rho\right|$. Let $R'=S\left[z_{0}\right]\cong S\left[Y_{0}\right]$.
By (\ref{enu:x^p-z_iPf}) and (\ref{enu:x^q-t_iPf}) we can use \ref{cla:SillyClaim}
and we have 
\begin{eqnarray*}
K & \supseteq & R'\left[z_{i},t_{l_{v}}^{v}\left|\, i<j,l_{v}<\rho\left(v\right)\right.\right]\cong\\
 &  & S\left[Y_{0}\right]\left[Y_{i},S_{l_{v}}^{v}\left|\,1<i<j,l_{v}<\rho\left(v\right)\right.\right]/I\left(j,\rho\right)=S\left(j,\rho\right)/I\left(j,\rho\right)
\end{eqnarray*}
 As desired.

Next we prove (\ref{enu:s-highPf}). Suppose $x$ is $q\mbox{-high}$
in $K$. So $x\in F\left(i,\rho\right)$ for some $i<\omega$ and
finite $\left|\rho\right|$. By Lemma \ref{lem:dimOfExtension}, $x$
is $q\mbox{-high}$ already in $F\left(i,\rho\right)$. Now apply
(\ref{enu:transExtenPf}) and Lemma \ref{lem:PurelyTrans}.

Next we prove (\ref{enu:p-highPf}). If $x\in F\left(\omega,\rho\right)$
is $\phigh$, then by Lemma \ref{lem:dimOfExtension}, $x$ is already
$\phigh$ in $F\left(\omega,\rho_{0}\right)$ where $\left|\rho_{0}\right|$
is finite. So it is enough to prove by induction on $\left|\rho\right|$
that the $\phigh$ elements of $F\left(\omega,\rho\right)$ are of
the form $c\cdot z_{i}^{m}$. 

\uline{The induction base}: Suppose $x\in F\left(\omega,1\right)$. 

If $x$ is already $\phigh$ in $F\left(z_{i}\right)$ for some $i$,
then by \ref{lem:PurelyTrans} $x\in F$.

Suppose now that $x$ is not $\phigh$ in any finite stage. Let $i<\omega$
be such that $x\in F\left(z_{i}\right)$, so there are relatively
prime polynomials $f_{0},g_{0}\in F\left[z_{i}\right]$, none of them
divisible by $z_{i}$, such that $x=\left(z_{i}\right)^{l_{0}}\frac{f_{0}\left(z_{i}\right)}{g_{0}\left(z_{i}\right)}$
for some $l_{0}\in\mathbb{Z}$. So it is enough to show that $u=x/\left(z_{i}\right)^{l_{0}}$
is $\phigh$ already in $F\left(z_{i}\right)$. So suppose not. Let
$X_{m}=\left\{ y\in F\left(\omega,1\right)\left|\, y^{p^{m}}=u\right.\right\} $.
Let $j<\omega$ be the first such that $X_{j}\subseteq F\left(z_{i}\right),X_{j+1}\nsubseteq F\left(z_{i}\right)$.
Let $s$ be the least natural number for which $X_{j+1}\subseteq F\left(z_{s}\right)$
($s>i$). Suppose $v\in X_{j+1}$, and let $v'=v^{p}\in X_{j}\subseteq F\left(z_{i}\right)$.
So $v'=\left(z_{i}\right)^{l_{1}}\frac{f_{1}\left(z_{i}\right)}{g_{1}\left(z_{i}\right)}$
where $f_{1},g_{1}$ are relatively prime, neither of them divisible
by $z_{i}$, $l_{1}\in\mathbb{Z}$. Since $\left(v'\right)^{p^{j}}=u$,
$l_{1}=0$. As $v^{p}\in F\left(z_{i}\right)$, by \ref{lem:Kummer},
we can write $v=\left(z_{s}\right)^{m}\cdot d$ for some $d\in F\left(z_{i}\right)$,
$m<\omega$, $p\nmid m$ (as $v\notin F\left(z_{s'}\right)$ for $s'<s$).
Denote $d=\left(z_{i}\right)^{l_{2}}\frac{f_{2}\left(z_{i}\right)}{g_{2}\left(z_{i}\right)}$
where $f_{2},g_{2}\in F\left[z_{i}\right]$ are relatively prime,
none of them divisible by $z_{i}$, $l_{2}\in\mathbb{Z}$. Since $v^{p}=v'$,
we have 
\[
\left(z_{s-1}\right)^{m}\cdot\left(z_{i}\right)^{l_{2}p}\left(\frac{f_{2}\left(z_{i}\right)}{g_{2}\left(z_{i}\right)}\right)^{p}=\left(\frac{f_{1}\left(z_{i}\right)}{g_{1}\left(z_{i}\right)}\right)
\]
and after raising to the power of $p$, $s-1-i$ times, we get $p^{s-i}l_{2}+m=0$,
so $p\mid m$ --- a contradiction.

\uline{The induction step}: Suppose we have $\rho^{+}$ and $\rho$
as before (i.e. $\rho^{+}\left(v\right)=\rho\left(v\right)+1$ for
some $v\in V$ and $\rho^{+}\left(w\right)=\rho\left(w\right)$ for
$w\neq v$) and $x\in F\left(\omega,\rho^{+}\right)$ is $\phigh$
there. Let $K'=F\left(\omega,\rho^{+}\right)$ and $L=F\left(\omega,\rho\right)$.
By (\ref{enu:x^q-t_iPf}), the degree of the extension $K'/L$ is
$q$. 

Denote by $N:K'\to L$ the norm of the extension. We use the following
properties of the norm: 
\begin{itemize}
\item Its multiplicative, and $N\left(a\right)=a^{q}$ for $a\in L$. 
\item If $K_{i}=F\left(i,\rho^{+}\right)$ and $L_{i}=F\left(i,\rho\right)$
then $N\upharpoonright K_{i}=N_{K_{i}}:K_{i}\to L_{i}$. 
\end{itemize}
$N\left(x\right)$ is $\phigh$ in $L$. So $y=x^{q}/N\left(x\right)$
is $\phigh$ in $K'$. Choose $i<\omega$ such that $x,y\in F\left(i+1,\rho^{+}\right)$.
We shall show that $y$ is $\phigh$ in $F\left(i+1,\rho^{+}\right)$.
Suppose that $u\in F\left(\omega,\rho^{+}\right)\backslash F\left(i+1,\rho^{+}\right)$
satisfies $u^{p}\in F\left(i+1,\rho^{+}\right)$ and $y$ is a $p^{m}$
power of $u$ for some $m<\omega$. Let $k=\max\left\{ n\left|\, u\notin F\left(n+1,\rho^{+}\right)\right.\right\} \geq i$
. By Lemma \ref{lem:Kummer}, as $u\in F\left(k+2,\rho^{+}\right)$
and $u^{p}\in F\left(k+1,\rho^{+}\right)$, we have $u=h\cdot\left(z_{k+1}\right)^{b}$
where $h\in F\left(k+1,\rho^{+}\right)$ and $0<b<p$. Hence $N\left(u\right)=N\left(h\right)\cdot N\left(z_{k+1}\right)^{b}=N\left(h\right)\cdot\left(z_{k+1}\right)^{bq}$.
Now, $N\left(h\right)\in F\left(k+1,\rho\right)$, so $N\left(u\right)\notin F\left(k+1,\rho\right)$
because by (\ref{enu:x^p-z_iPf}) $z_{k+1}\notin F\left(k+1,\rho\right)$
and $\left(p,bq\right)=1$. On the other hand, $N\left(y\right)=N\left(u\right)^{p^{m}}$
and $N\left(y\right)=N\left(x^{q}/N\left(x\right)\right)=N\left(x\right)^{q}/N\left(x\right)^{q}=1$,
so $N\left(u\right)$ is algebraic over $F$, which is a contradiction
to (\ref{enu:transExtenPf}). 

By Lemma \ref{lem:PurelyTrans}, $y\in F$ and is $\phigh$ there,
therefore $y\cdot N\left(x\right)=x^{q}$ is $\phigh$ in $L$. By
the induction hypothesis, $x^{q}$ has the form $c\cdot\left(z_{i}\right)^{m}$,
hence $x^{q}\in F\left(z_{i}\right)$. By (\ref{enu:NormalFormPf}),
we get the equation:
\[
c\cdot\left(z_{i}\right)^{m}=x^{q}=d^{q}\left(\frac{f^{q}\left(z_{i}\right)}{g^{q}\left(z_{i}\right)}\right)\prod_{w\in W}\left(T_{w}\left(z_{0}\right)\right)^{l_{w}}
\]
 for some finite $W\subseteq V$, $0<l_{w}<q$. This implies $g=1$,
$q\mid m$, $f\left(z_{i}\right)=\left(z_{i}\right)^{m/q}$, and $W=\emptyset$.
Hence $x=\varepsilon\cdot\left(d\cdot\left(z_{i}\right)^{m/q}\right)$
where $\varepsilon^{q}=1$ (so $\varepsilon\in F$) as promised.

Clause (\ref{enu:noMoreRootsPf}) follows from the previous clauses:
if $x^{p'}=z_{0}$, then: if $p'=q$ then by (\ref{enu:NormalFormPf})
$x=c\cdot\frac{f\left(z_{0}\right)}{g\left(z_{0}\right)}\prod_{v\in W'}\left(t_{1}^{v}\right)^{l_{v}}$
($0<l_{v}<q$), so $z_{0}=x^{q}=c\cdot\left(\frac{f\left(z_{0}\right)}{g\left(z_{0}\right)}\right)^{q}\prod_{v\in W}\left(T_{v}\right)^{l_{v}}$,
so $W=\emptyset$, $g=1$, and we easily derive a contradiction. If
$p'\neq q$, use Lemma \ref{lem:dimOfExtension}.

Clause (\ref{enu:CardinalityPf}): one defines by induction on $\left|\rho\right|,n$
an injective function $\varphi_{n,\rho}:F\left(n,\rho\right)\to F^{\left\langle <\omega\right\rangle }$
such that $\varphi_{n,\rho}\subseteq\varphi_{n',\rho'}$ whenever
$n\leq n'$ and $\rho\leq\rho'$ (i.e. $\rho\left(v\right)\leq\rho'\left(v\right)$
for all $v\in V$). Why is this enough? we shall need:
\begin{prop*}
for all $1\leq m,n<\omega$, $\rho,\rho'\in^{V}\omega$, \\
$F\left(n,\rho\right)\cap F\left(m,\rho'\right)=F\left(\min\left(n,m\right),\min\left(\rho,\rho'\right)\right)$
where $\min\left(\rho,\rho'\right)\left(v\right)=\min\left(\rho\left(v\right),\rho'\left(v\right)\right)$. \end{prop*}
\begin{proof}
The proof is an argument similar to the one used to prove (\ref{enu:transExtenPf})
and Lemma \ref{lem:Intersection}.

Assume $x\in F\left(n,\rho\right)\cap F\left(m,\rho'\right)$. Assume
that $n,\left|\rho\right|$ is minimal with respect to $x\in F\left(n,\rho\right)$.
If $\left(n,\rho\right)\leq\left(m,\rho'\right)$ then we are done.
If not, suppose $m<n$ (the case where $\rho\not\leq\rho'$ is similar).
So $x\notin F\left(m,\rho\right)$. Since $x\in F\left(m,\rho'\right)$,
we can find $\rho\leq\rho_{1},\rho_{2}$ such that $\rho_{1}\left(v\right)=\rho_{2}\left(v\right)$
for all $v\neq v_{0}$ but $\rho_{2}\left(v_{0}\right)=\rho_{1}\left(v_{0}\right)+1$
and $x\in F\left(m,\rho_{2}\right)\backslash F\left(m,\rho_{1}\right)$
and then $F\left(m,\rho_{2}\right)=F\left(m,\rho_{1}\right)\left(x\right)$,
so also $F\left(n,\rho_{2}\right)=F\left(n,\rho_{1}\right)\left(x\right)$
but since $x\in F\left(n,\rho_{1}\right)$, we get that $F\left(n,\rho_{1}\right)=F\left(n,\rho_{2}\right)$
and this contradicts (\ref{enu:x^q-t_iPf}).
\end{proof}
By this proposition, $\bigcup\left\{ \varphi_{n,\rho}\left|\, n\in\omega,\rho\in^{V}\omega\right.\right\} $
will be an injective function from $K$ to $F^{\left\langle <\omega\right\rangle }$. 

For the construction, one should use the fact that we can represent
the sequence $\rho$ as a function from polynomials to $\omega$,
hence it has a code in $F^{\left\langle <\omega\right\rangle }$.
So the idea is that given $x$ with minimal $\left(n,\rho\right)$
such that $x\in F\left(n,\rho\right)$, code $x$ as $\left(n,\rho\right)$
and then for each choice of $\left(n',\rho'\right)$ such that $\left(n',\rho'\right)<\left(n,\rho\right)$
with difference exactly one (either $n'=n-1$ or $\rho'\left(v\right)=\rho\left(v\right)-1$
for some $v$), use the code we already have for $F\left(n',\rho'\right)$
and the representation of $x$ as linear combination of $\left(z_{n-1}\right)^{i}$,
$i<p$ or $\left(t_{\rho\left(v\right)-1}^{v}\right)^{i}$, $i<q$
. 

Now for clause (\ref{enu:FAnyFieldPf}):

Assume then, that $F$ is some field, not necessary containing any
roots of unity. Let $\bar{F}$ be its algebraic closure. The lemma
works for $\bar{F}$ because $z_{0}$ is transcendental over $F$
hence over $\bar{F}$ and the conditions on the polynomials $T_{v}$
still hold. Denote by $K'$ the field corresponding to it. So $K\subseteq K'$,
and for every $n,\rho$, $F\left(n,\rho\right)\subseteq\bar{F}\left(n,\rho\right)$.
(\ref{enu:x^p-z_iPf}) and (\ref{enu:x^q-t_iPf}) are clearly true
(for $F$) as they are true for $\bar{F}$.

Hence, (\ref{enu:transExtenPf}) is true as well: the proof uses only
(\ref{enu:x^q-t_iPf}). (\ref{enu:transExtenPf}) implies that $K\cap\bar{F}=F$,
and this allows us to prove all the other clauses, for example ---
(\ref{enu:p-highPf}) --- If $x$ is $\phigh$ in $K$ then it is
$\phigh$ in $K'$ hence it has the form $c\cdot\left(z_{i}\right)^{m}$
for $c\in\bar{F}$, but then $c\in\bar{F}\cap K=F$.

This completes the proof of this lemma. 
\end{proof}
\bibliographystyle{alpha}
\bibliography{common}

\end{document}